\documentclass[10pt]{article}
\usepackage{graphicx,epsfig,color,amsmath,amssymb,amsthm,fullpage,caption,mathtools,subcaption,url} 
\usepackage{wrapfig}
\usepackage{cleveref}
\usepackage[ruled]{algorithm2e}

\newcommand{\real}{\mathbb{R}}
\newcommand{\vo}[1]{\boldsymbol{#1}}
\newcommand{\mo}[1]{\boldsymbol{#1}}
\newcommand{\A}{\vo{A}}
\newcommand{\B}{\vo{B}}
\newcommand{\x}{\vo{x}}
\newcommand{\F}{\vo{F}}
\newcommand{\EE}{\vo{E}}
\renewcommand{\L}{\vo{L}}

\newcommand{\fhat}{\vo{\hat{f}}}
\newcommand{\fhatp}{\vo{\hat{f}}(\param)}
\newcommand{\fhatpt}{\vo{\hat{f}}^T(\param)}
\newcommand{\fp}{\vo{f}(\param)}
\newcommand{\fpt}{\vo{f}^T(\param)}

\newcommand{\phip}{\vo{\Phi}(\param)}
\newcommand{\phipt}{\vo{\Phi}^T(\param)}

\newcommand{\xdot}{\dot{\vo{x}}}

\newcommand{\param}{\vo{\Delta}}

\newcommand{\pdfp}{p(\param)}
\newcommand{\set}[1]{\mathcal{#1}}
\newcommand{\Exp}[1]{\mathbb{E}\left[#1\right]}
\newcommand{\E}[1]{\Exp{#1}}

\newcommand{\basis}[2]{\phi_{#1}\left(#2\right)}
\newcommand{\I}[1]{\vo{I}_{#1}}
\newcommand{\X}{\vo{X}}

\newcommand{\W}{\vo{W}}

\newcommand{\R}{\vo{R}}
\newcommand{\Q}{\vo{Q}}

\newcommand{\Phin}[1]{\vo{\Psi}(\param)}
\newcommand{\Phint}[1]{\vo{\Psi}^T(\param)}

\newtheorem{theorem}{Theorem}

\theoremstyle{definition}
\newtheorem{remark}{Remark}

\newcommand{\Real}{\mathbb R}

\renewcommand{\vec}[1]{\textbf{vec}\left(#1\right)}
\newcommand{\xpc}{\x_{\text{pc}}}
\newcommand{\xgpc}{\x_{\text{gpc}}}
\newcommand{\xcpc}{\x_{\text{cpc}}}

\newcommand{\xpcdot}{\xd_{\text{pc}}}
\newcommand{\inner}[1]{\left\langle \vo{e}\phi_i\right\rangle}

\newcommand{\eqnlabel}[1]{\label{eqn:#1}}
\newcommand{\figlabel}[1]{\label{fig:#1}}

\newcommand{\eqn}[1]{(\ref{eqn:#1})}
\newcommand{\Eqn}[1]{Eq.(\ref{eqn:#1})}
\newcommand{\fig}[1]{fig.(\ref{fig:#1})}
\newcommand{\Fig}[1]{Fig.(\ref{fig:#1})}

\DeclareMathAlphabet{\mathbfsf}{\encodingdefault}{\sfdefault}{bx}{n}

\newcommand{\C}{\vo{C}}
\newcommand{\D}{\vo{D}}

\renewcommand{\H}{\vo{H}}

\newcommand{\Ap}{\A(\param)}

\newcommand{\xd}{\dot{\x}}
\newcommand{\domain}[1]{\set{D}}
\newcommand{\domainD}{\set{D}_{\param}}

\newcommand{\diag}{\textbf{diag}}

\newcommand{\trace}[1]{\textbf{tr}\left( #1 \right)}
\newcommand{\tr}[1]{\textbf{tr}\:#1}

\newcommand{\algo}[1]{Algorithm \ref{#1}}
\newcommand{\lfp}[1]{linear-fit-predict }

\title{On Improved Statistical Accuracy of Low-Order Polynomial Chaos Approximations}
\author{\bf \Large Vedang M. Deshpande \footnote{vedang.deshpande@tamu.edu} \;\;\; Raktim Bhattacharya \footnote{raktim@tamu.edu} \\[2mm]  \textit{Intelligent Systems Research Laboratory}, \\ \url{isrlab.github.io}\\[2mm]Aerospace Engineering, Texas A\&M University, \\  College Station, TX, 77843-3141.}
\date{}

\begin{document}
\maketitle

\begin{abstract}
Polynomial chaos expansion is a popular way to develop surrogate models for stochastic systems with arbitrary random variables. Standard techniques such as Galerkin projection, stochastic collocation, and least squares approximation, are applied to determine polynomial chaos coefficients, which define the surrogate model.
Since the surrogate models are developed from a function approximation perspective, there is no reason to expect accuracy of statistics from these models.  The statistical moments estimated from the surrogate model may significantly differ from the true moments, especially for lower order approximations.
Often arbitrary high orders are required to recover, for example, the second moment. In this paper, we present modifications of standard techniques and determine polynomial chaos coefficients by solving a constrained optimization problem. We present this new approach for algebraic functions and differential equations with random parameters, and demonstrate that the surrogate models from the new approach are able to recover the first two moments exactly.
\end{abstract}

\section{Introduction} \label{sec:intro}
Polynomial chaos (PC) expansions have been widely studied to model uncertainties and represent stochastic processes. The basic idea is to represent an arbitrary random variable or a function of random variable as an expansion of polynomials of random variables with known distributions. In practice, this expansion is truncated at a finite number of terms, and the polynomial basis is generally chosen to be orthogonal.
In \cite{wiener_askey}, authors presented a correspondence between random distributions and optimal basis from the Askey family of orthogonal polynomials to achieve exponential convergence of the error with respect to the order of approximation, termed as generalized polynomial chaos (gPC).
Once the polynomial basis is selected, Galerkin projection (GP) approach is used to determine the coefficients associated with polynomials in the expansion. Thus, an infinite dimensional stochastic system is approximated by a finite dimensional deterministic system, commonly known as \textit{surrogate model} of the original system.
A surrogate model is typically employed when the original system is complex and computationally expensive. Surrogate models greatly reduce the complexity and can be evaluated much faster than the actual system. For example, instead of sampling from a complex arbitrary random distribution, it is represented by a polynomial expansion of uniform or Gaussian random variables, which can be sampled easily using built-in computer algorithms.
PC expansions are also used in solving stochastic differential equations. A system of deterministic differential equations in terms of PC coefficients is derived using the GP approach, and then it is solved using standard deterministic techniques.
Effectiveness of gPC for numerical solution of stochastic differential equations has been investigated in \cite{dXiu_jcp, dXiu_bc}.

The GP approach may involve complex integrals as it projects the error against each basis function using inner products. These integrals can be evaluated using intrusive as well as non-intrusive techniques. However, non-intrusive methods are more practical for real world problems. Non-intrusive methods such as collocation or least squares approximation use computer simulations as black box to generate a set of responses that is used to calculate PC coefficients using the least squares solution. A comparative study of various non-intrusive methods for gPC has been presented in \cite{constantine,sc2009eldred}.

Theoretically, if we let the number of terms in a PC expansion to be infinite, it will converge to any stochastic process with a finite second order moment \cite{wiener_askey, cameron_martin}. However, for a practical application, an expansion is truncated at finite number of terms, which is a major source of error in surrogate models. Consequently, statistical moments estimated from these models can be very erroneous. As discussed in the following section, a surrogate model obtained using GP approach recovers the first moment with no error. However, errors in higher moments can be quite large, especially for lower order approximations.
Models from stochastic collocation (SC) and least squares (LS) approximation are unable to accurately recover any moment, with latter being the least accurate. See \fig{pcErrors} as an illustration of this.

Therefore, to improve the statistical accuracy of low order PC approximations, we present a framework to calculate PC coefficients such that estimated first two moments match exactly with the true moments. The framework presented in this paper is a modification of standard GP and LS approaches, referred herein as \textit{constrained} $\mathcal{L}_2$  and \textit{constrained} $l_2$ formulations respectively. We also show that such corrections are not possible in standard SC methods, since SC is posed as an interpolation problem and there are no degrees of freedom to satisfy additional constraints.

The rest of the paper is organized as follows. \Cref{sec:framework} presents the constrained $\mathcal{L}_2$  and constrained $l_2$ formulations, \textit{for algebraic functions}, in the form of \Cref{thm:gp1,thm:bestL2,thm:bestl2}, which are main results of this paper. In \Cref{sec:lfp}, we develop an algorithm, which is built on results from \Cref{sec:framework}, to develop a linear surrogate model from time series data. The linear model approximates the data in the $l_2$ optimal sense, and also recovers the first two moments of the data. Numerical results obtained from the proposed algorithm are presented in \Cref{sec:num_results}. Concluding remarks are discussed in \Cref{sec:conclusions}. We also show, in \Cref{sec:gpUnbiased}, that GP approach leads to surrogate models with unbiased approximation error that has minimum variance.

\section{Approximation of algebraic functions of random variables} \label{sec:framework}
Let $\param\in\domainD\subseteq\real^d$ be a random vector with probability density function $\pdfp$. Let $\vo{f}(\param):\domainD\mapsto\real^n$ be a vector function whose components are square integrable, i.e.
$$
\int_{\domainD} f_i^2(\param)\pdfp d\param < \infty, \quad i=1,2,\cdots n.
$$
The objective here is to approximate $\vo{f}(\param)$ using PC expansions, i.e. find PC coefficients $\vo{f}_i\in\real^n$ such that
\begin{align}
\vo{f}(\param) \approx \fhat(\param) :=  \sum_{i=0}^N \vo{f}_i\phi_i(\param), \eqnlabel{pcF}
\end{align}
for a known set of basis polynomials $\{\phi_i(\param)\}$, which is selected based on the type of random distribution \cite{wiener_askey}.
The highest degree of polynomials $\{\phi_i(\param)\}$ is called the \textit{order of approximation}, and is denoted by $\kappa$. The number of basis functions ($N+1$) is related to the dimension ($d$) of the random vector $\param$, and the order of approximation $\kappa$, by
\begin{align*}
(N+1)=\frac{(d+\kappa)!}{d! \: \kappa !}.
\end{align*}
We define $\mo{\Phi}(\param)$ as
\begin{align}
\mo{\Phi}(\param) &:= \begin{bmatrix}\basis{0}{\param}, & \cdots, & \basis{N}{\param}\end{bmatrix}^T.
\eqnlabel{Phi:def}
\end{align}
We also define matrix $\F\in\real^{n\times(N+1)}$, with polynomial chaos coefficients $\vo{f}_i$, as its columns,
\[ \F := \begin{bmatrix} \vo{f}_0, & \cdots, & \vo{f}_N \end{bmatrix}.\]
Therefore, $\fhat(\param)$ can be compactly written as
\begin{align}
\fhat(\param) = \F\mo{\Phi}(\param) \eqnlabel{compactFpc}.
\end{align}
The coefficients $\vo{f}_i$ are determined using approaches like GP, SC or LS. We next describe the errors in statistical moments, associated with these three approaches, and present new formulations with better accuracy.

\subsection{Constrained $\mathcal{L}_2$-optimal approximation}
In the GP formulation, or the $\mathcal{L}_2$ formulation, the basis functions $\phi_i(\param)$ are $d$-variate polynomials that are orthogonal with respect to $\pdfp$, i.e.,
\begin{align}
\E{\phi_i(\param)\phi_j(\param)} := \int_{\domainD} \phi_i(\param)\phi_j(\param) \pdfp d\param= 0, \text{ for } i \neq j.
\eqnlabel{orthoBasis}
\end{align}
Since $\phi_0(\param) = 1$ for orthogonal basis, it follows that $\E{\vo{\Phi}(\param)} = [1,0,\cdots,0]^T$.

Optimal coefficients $\vo{f}_i$ are determined by projecting the error $\vo{e}(\param):=\vo{f}(\param) - \fhat(\param)$ against each basis polynomial and setting it to zero,
$$
\E{\vo{e}(\param) \phi_i(\param)} = \vo{0}, \text{ for } i=0,\cdots,N,
$$
or more compactly
\begin{align}
\E{\vo{f}(\param)\vo{\Phi}^T(\param)} - \F\E{\mo{\Phi}(\param)\vo{\Phi}^T(\param)} = \vo{0}.
\eqnlabel{gp:errProj}
\end{align}
From orthogonality \eqn{orthoBasis},
\begin{align}
\E{\mo{\Phi}(\param)\vo{\Phi}^T(\param)} = \diag\begin{pmatrix}\E{\phi^2_0(\param)} &\cdots& \E{\phi^2_N(\param)}\end{pmatrix} =: \W.\eqnlabel{gp:W}
\end{align}
Let us denote the optimal coefficients matrix determined using GP formulation by $\F_{\text{GP}}$ which can be solved as
\begin{align}
\F_\text{GP} = \E{\vo{f}(\param)\vo{\Phi}^T(\param)}\W^{-1}. \eqnlabel{pcAnalytical}
\end{align}
The first and second order moments of $\fhatp$ are computed as follows,
\begin{align}
\E{\fhatp} & = \F\E{\mo{\Phi}(\param)} = \F\begin{bmatrix}1 & 0 & \cdots & 0\end{bmatrix}^T,  \eqnlabel{gp:meanFhat} \\
\E{\fhatp\fhatpt} & = \F\E{\mo{\Phi}(\param)\mo{\Phi}^T(\param)}\F^T = \F\W\F^T. \eqnlabel{gp:covFhat}
\end{align}
Higher order moments can be similarly computed from $\fhatp$, and will have approximation errors.
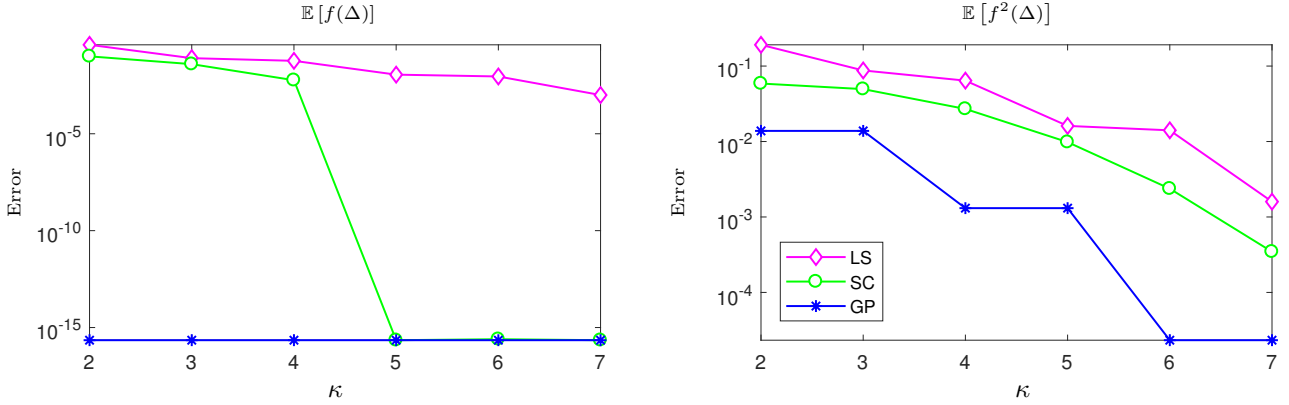
\begin{figure}[ht!]
\includegraphics[width=\textwidth]{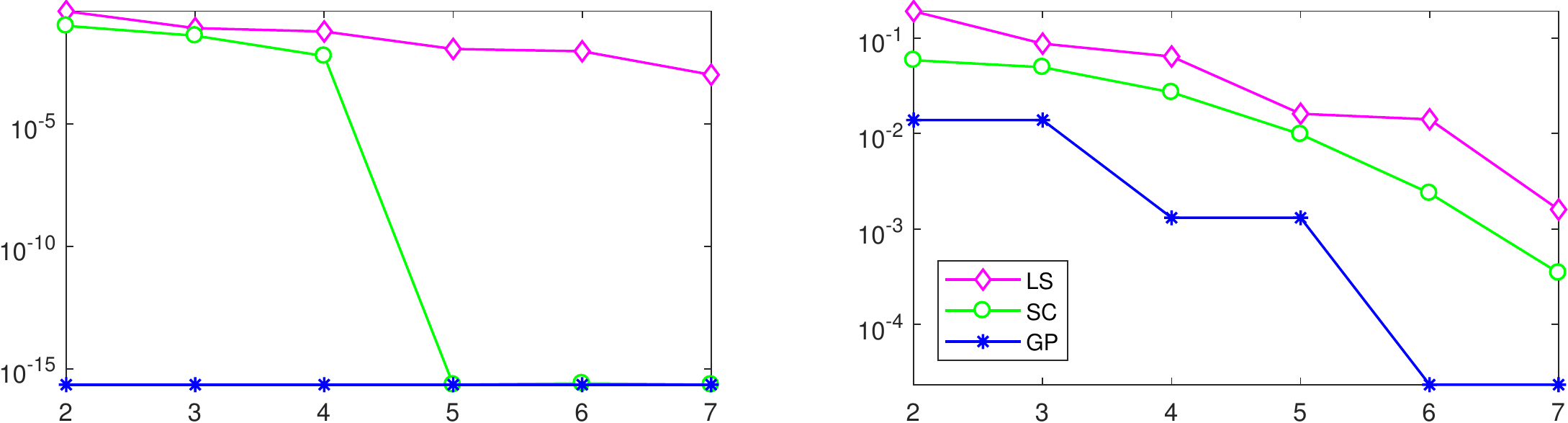}
\begin{picture}(0,0)
        \put(-11,70){\rotatebox{90} {{\scriptsize Error}}} 
        \put(240,70){\rotatebox{90} {{\scriptsize Error}}} 
        \put(110,2){{$\kappa$}}
        \put(370,2){{$\kappa$}}
        \put(100,145){{\scriptsize $\E{f(\Delta)}$}}
        \put(350,145){{\scriptsize $\E{f^2(\Delta)}$}}
\end{picture}
\caption{\small Absolute errors in estimated moments of $f(\Delta)=\Delta^8$ using LS, SC, and GP formulations for different approximation orders ($\kappa$), where $\Delta$ is uniformly distributed over $[-1,1]$.}
\figlabel{pcErrors}
\end{figure}

In \fig{pcErrors}, we compare the absolute errors in first two moments of $f(\Delta):=\Delta^8$ estimated from polynomial chaos expansion of different orders ($\kappa$) w.r.t. true moments that are calculated analytically. The errors are plotted on a semi-log scale, lower bounded by $2^{-52}$, which is the machine precision.
The mean computed from the GP formulation is exact (zero error) for all orders of approximation. However, second moment estimated using GP formulation is erroneous.
The SC and LS methods which are discussed later, perform poorly in estimating both the moments as compared to the GP formulation. Among these three approaches under consideration, the LS method performs the worst. As expected, we observe that the errors in moments decrease as the approximation order is increased.

The exactness of the mean estimated using GP formulation can be verified by examining \eqn{pcAnalytical}. Since $\phi_0(\param)=1$, it is easy to check that the first column of $\F_\text{GP}$ is equal to $\E{\fp}$.
Therefore, in the $\mathcal{L}_2$ framework there are no errors in computing the mean of algebraic expressions, for any order of approximation.
However, the error in second and higher order moments can be arbitrarily large. The objective here is to develop a formulation to derive $\F$, such that the errors in certain specified higher order moments are zero, or minimized.

We next present the result for determining PC coefficients that exactly match first and second order statistics of $\fp$.

\begin{theorem}\label{thm:gp1}
The PC coefficients that provide exact first and second order statistics are given by the columns of
\begin{align}
\F = \begin{bmatrix} \E{\fp} &  \vo{L}\vo{U}\W_1^{-1/2} \end{bmatrix}, \eqnlabel{gp:twoOrderExact}
\end{align}
where
\begin{align}
\W_1 &= \mathbf{diag}\begin{pmatrix}\E{\phi^2_1(\param)} &\cdots& \E{\phi^2_N(\param)}\end{pmatrix},\\
\vo{L}\vo{L}^T &= \E{\fp\fpt} -\E{\fp}\E{\fpt}, \eqnlabel{gp:L}
\end{align}
and $\vo{U}\in\real^{n\times N}$ is an arbitrary matrix satisfying $\vo{U}\vo{U}^T = \I{n}$.
\end{theorem}

\begin{proof}
The error in the second moment can be set to zero, if we constrain
\begin{align}
\F\W\F^T = \E{\fp\fpt}. \eqnlabel{gp:quadConstr}
\end{align}
Let us partition,
\begin{align}
\F := \begin{bmatrix} \E{\fp} & \F_1 \end{bmatrix}, \eqnlabel{gp:partF}
\end{align}
where $\F_1\in\real^{n\times N}$. This satisfies the mean constraint directly. Partitioning $\W$ as
$$
\W := \begin{bmatrix} \E{\phi_0(\param)} & \vo{0}\\\vo{0} & \W_1 \end{bmatrix} =  \begin{bmatrix} 1 & \vo{0}\\\vo{0} & \W_1 \end{bmatrix},
$$
\eqn{gp:quadConstr} becomes
\begin{align}
\F_1\W_1\F_1^T = \E{\fp\fpt} -\E{\fp}\E{\fpt}. \eqnlabel{gp:quadConstr1}
\end{align}
Since $\E{\fp\fpt} -\E{\fp}\E{\fpt} \ge 0$, we employ Cholesky factorization to write
\begin{align*}
\vo{L}\vo{L}^T &= \E{\fp\fpt} -\E{\fp}\E{\fpt}.
\end{align*}
\Eqn{gp:quadConstr1} is a standard linear algebra problem with solution
\begin{align}
\F_1 = \vo{L}\vo{U}\W_1^{-1/2},\eqnlabel{gp:soln}
\end{align}
and $\vo{U}\in\real^{n\times N}$ is an arbitrary (rectangular) unitary matrix, i.e., $\vo{U}\vo{U}^T = \I{n}$.  Therefore, the PC coefficients that provide first and second moments with no error are given by the columns of
\begin{align*}
\F = \begin{bmatrix} \E{\fp} &  \vo{L}\vo{U}\W_1^{-1/2} \end{bmatrix}.
\end{align*}
\end{proof}

The solution of $\F$ in \eqn{gp:twoOrderExact} is quite different from the solution in \eqn{pcAnalytical} that is obtained via GP approach.
The variable $\vo{U}$ parameterizes a family of solutions for $\F$ that exactly recovers the first and second moments. Clearly, for rows of $\vo{U}$ to be orthonormal, we require $N\ge n$.

If only first and second order statistics are required, then we can choose $N=n$ and $\vo{U}=\I{n}$, and that results in
\begin{align*}
\F = \begin{bmatrix} \E{\fp} &  \vo{L}\W_1^{-1/2} \end{bmatrix}. 
\end{align*}
For $N>n$, there are extra degrees of freedom that can be used to minimize other errors.

In the standard $\mathcal{L}_2$ formulation the coefficients $\F$ are determined such that $\E{\|\fp-\fhatp\|_2}$ is minimized, and results in \eqn{gp:errProj}. The number of equations in \eqn{gp:errProj} is equal to the number of unknowns, which results in a unique solution for $\F$.
Adding constraints for first and second order statistics will result in more constraints than variables, and thus \eqn{gp:errProj} cannot be exactly satisfied. Therefore, optimal coefficients can be obtained via a constrained minimization of the residual error, i.e.,
\begin{align}
\min_{\vo{U}} \E{\vo{e}^T(\param)\vo{e}(\param)}, \text{ subject to } \vo{U}\vo{U}^T=\I{n}. \eqnlabel{QCopt}
\end{align}
Note that $\E{\vo{e}^T(\param)\vo{e}(\param)} = \tr{\E{\vo{e}(\param)\vo{e}^T(\param)}},$ and
\begin{align*}
\E{\vo{e}(\param)\vo{e}^T(\param)} & = \E{\Big(\fp-\F\mo{\Phi}^T(\param)\Big)\Big(\fp-\F\mo{\Phi}^T(\param)\Big)^T},\\
& = \vo{Q} - \F_1\R^T - \R\F_1^T + \F_1\W_1\F_1^T,
\end{align*}
where
\begin{align}
\vo{Q} & := \E{(\fp-\E{\fp})(\fp-\E{\fp})^T},\\
\R & := \E{\fp\vo{\Phi}^T_{1}(\param)}, \eqnlabel{R}
\end{align}
$\F_1$ depends on $\vo{U}$ as given by \eqn{gp:soln}, and $\vo{\Phi}_{1}(\param)$ is the sub-vector of $\vo{\Phi}(\param)$ without the first element, i.e.,
$$
\vo{\Phi}_{1}(\param) := \begin{bmatrix}\phi_1(\param) \\ \vdots \\ \phi_N(\param)\end{bmatrix}.
$$
Therefore, the optimization problem in \eqn{QCopt} can be written as
\begin{align}
\min_{\vo{U}\in\real^{n\times N}} \trace{\vo{Q} - \F_1\R^T - \R\F_1^T + \F_1\W_1\F_1^T}, \text{ subject to } \vo{U}\vo{U}^T=\I{n}. \eqnlabel{nonConvex1}
\end{align}
The optimization problem in \eqn{nonConvex1} is non convex due to the constraint $\vo{U}\vo{U}^T=\I{n}$. However, it is a quadratically constrained quadratic programming problem, which can be converted to a convex optimization problem using various relaxations techniques \cite{bao2011semidefinite, anstreicher2012convex, park2017general} and solved with existing solvers \cite{cvxpy,cvxpy_rewriting}.

In this paper, we propose a new formulation to solve the constrained nonconvex optimization problem analytically. In the new approach, the solution from \eqn{pcAnalytical} is projected on to the constraint set $\vo{U}\vo{U}^T=\I{n}$. The new formulation is formally presented as the following theorem.

\begin{theorem}
\label{thm:bestL2}
The coefficients that result in the $\mathcal{L}_2$-optimal PC approximation, subject to constraints on first and second-order statistics, is given by
\begin{align}
\F := \begin{bmatrix} \E{\fp} &  \vo{L}\vo{U}^\ast\W_1^{-1/2} \end{bmatrix},
\end{align}
where $\vo{U}^\ast :=  \vo{M}_1\vo{T}\vo{M}^T_2$, $\vo{M}_1$ and $\vo{M}_2$ are unitary matrices obtained from the singular value decomposition of $\vo{L}^{-1}\R\W_1^{-1/2}$, i.e.,
$$
\vo{L}^{-1}\R\W_1^{-1/2} = \vo{M}_1\vo{D}\vo{M}^T_2,
$$
and  $\vo{T}:=\begin{bmatrix}\I{n} & \vo{0}_{n\times(N-n)}\end{bmatrix}$.
\end{theorem}

\begin{proof}
The optimal solution from the GP approach results in $\F_\text{GP}$ from \eqn{pcAnalytical}. The corresponding partitioned $\F_1$ is $\F_{1\text{GP}}:=\R\W_1^{-1}$. The optimal cost associated with this solution is denoted by $J_\text{GP}$ and is given by
$$
J_\text{GP} :=  \tr{\E{\vo{e}_\text{GP}(\param)\vo{e}^T_\text{GP}(\param)}} = \trace{\Q - \R\W_1^{-1}\R^T}.
$$

Next, the cost-function, $J_{\vo{U}}$, corresponding to $\F$ given by \eqn{gp:twoOrderExact}, for any $\vo{U}\in\real^{n\times N}$ is given by
\begin{align*}
J_{\vo{U}} &:=  \tr{\E{\vo{e}_{\vo{U}}(\param)\vo{e}_{\vo{U}}(\param)}},\\
& = \trace{\Q + \vo{L}\vo{U}\vo{U}^T\vo{L}^T - \R\W_1^{-1/2}\vo{U}^T\vo{L}^T - \vo{L}\vo{U}\W_1^{-1/2}\R^T}.
\end{align*}
The difference in costs $J_{\vo{U}}$ and $J_\text{GP}$ is
\begin{align}
\notag J_{\vo{U}} - J_\text{GP} &=  \trace{\vo{L}\vo{U}\vo{U}^T\vo{L}^T - \R\W_1^{-1/2}\vo{U}^T\vo{L}^T - \vo{L}\vo{U}\W_1^{-1/2}\R^T + \R\W_1^{-1}\R^T},\\
& = \trace{\left(\vo{L}\vo{U}-\R\W_1^{-1/2}\right)\left(\vo{L}\vo{U}-\R\W_1^{-1/2}\right)^T}.\eqnlabel{gp:U>GP}
\end{align}
Therefore, $J_{\vo{U}} - J_\text{GP} = 0$ is achieved for $\vo{U} =  \vo{U}_\text{GP}$, where
\begin{align*}
\vo{U}_\text{GP}&:=\vo{L}^{-1}\R\W_1^{-1/2}.
\end{align*}
However, in general $\vo{U}_\text{GP}\vo{U}^{T}_\text{GP} \neq \I{n}$. Therefore, we project $\vo{U}_\text{GP}$ on the constraint set $\vo{U}\vo{U}^T = \I{n}$, by first expressing $\vo{U}_\text{GP}$ as its singular-value decomposition, i.e.,
$\vo{U}_\text{GP} = \vo{M}_1\vo{D}\vo{M}^T_2$, where $\vo{D}=\begin{bmatrix}\vo{\Lambda} & \vo{0}_{n\times(N-n)}\end{bmatrix}$, and $\vo{\Lambda}$ is diagonal matrix with singular values of $\vo{U}_\text{GP}$. Optimal $\vo{U}^\ast$ subject to $\vo{U}^\ast\vo{U}^{\ast^T} = \I{n}$, is recovered as $\vo{U}^\ast := \vo{M}_1\vo{T}\vo{M}^T_2$, where  $\vo{T}:=\begin{bmatrix}\I{n} & \vo{0}_{n\times(N-n)}\end{bmatrix}$.
\end{proof}
Note, \eqn{gp:U>GP} implies that $J_{\vo{U}} \geq J_\text{GP}$ for any $\vo{U}\neq\vo{U}_\text{GP}$. Therefore, enforcing the second-moment constraint increases the approximation error in the $\mathcal{L}_2$ sense.

\begin{figure}[ht!]
\begin{subfigure}[b]{\textwidth}
\includegraphics[width=\textwidth]{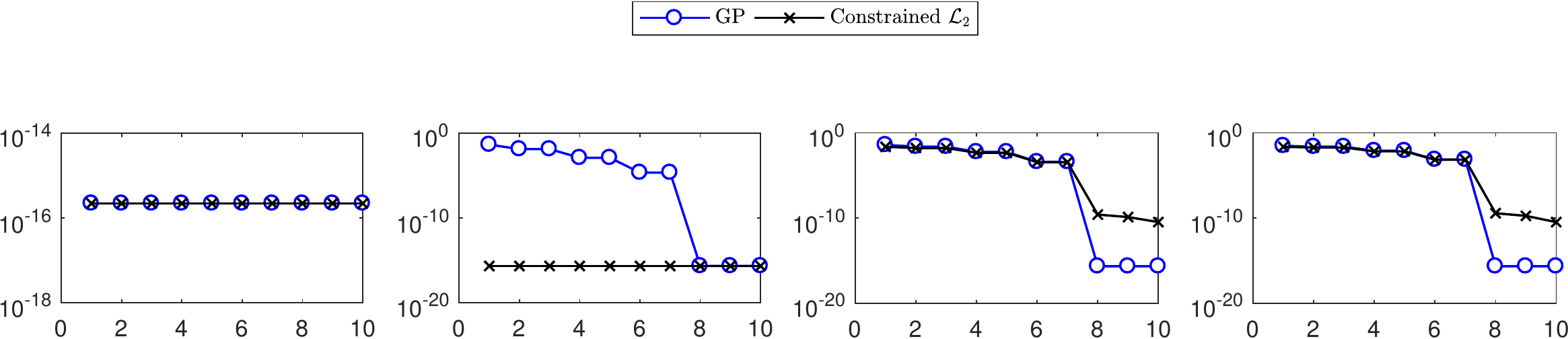}
\caption{{\small $f(\Delta) := \Delta^8$.}}
\figlabel{gp1}
\end{subfigure}\vspace{2mm}
\begin{subfigure}[b]{\textwidth}
\includegraphics[width=\textwidth]{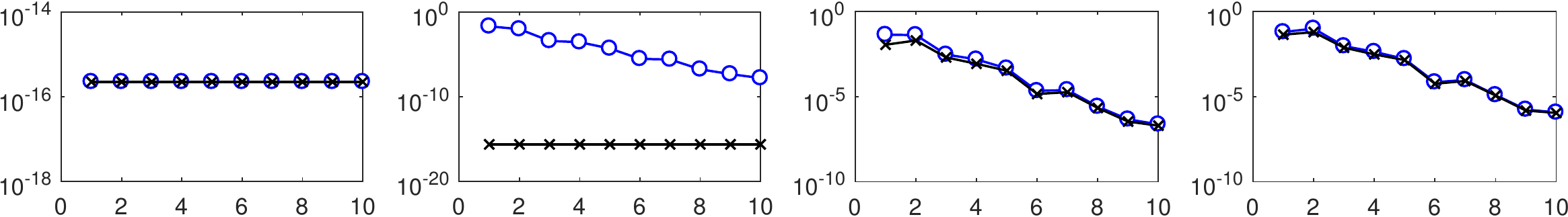}
\caption{{\small $f(\Delta) := \frac{1}{1+\Delta+\Delta^2}$.}}
\figlabel{gp2}
\end{subfigure}\vspace{2mm}
\begin{subfigure}[b]{\textwidth}
\includegraphics[width=\textwidth]{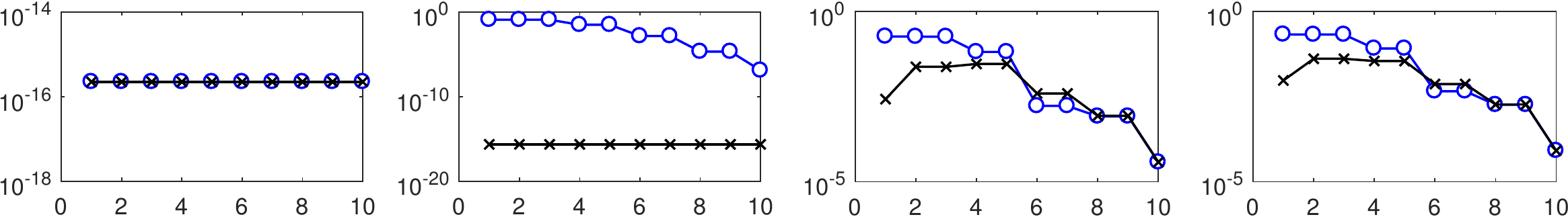}
\caption{{\small $f(\Delta) := \sin^2(3\Delta)$.}}
\figlabel{gp3}\vspace{1em}
\end{subfigure}
\begin{subfigure}[b]{\textwidth}
\includegraphics[width=\textwidth]{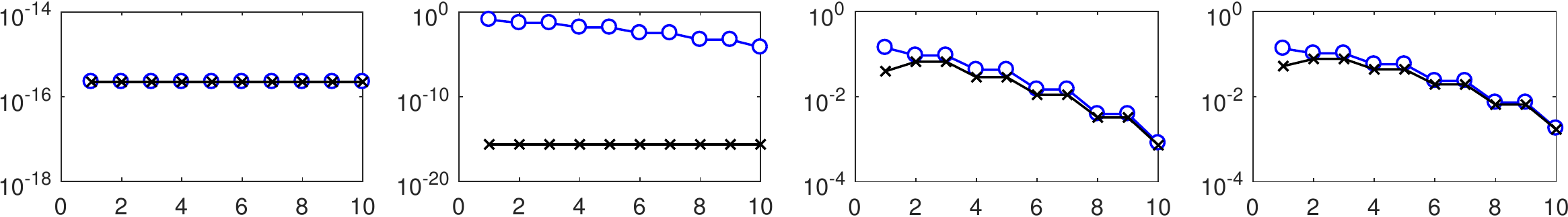}
\begin{picture}(0,0)
    \put(-8,40){\rotatebox{90} {{\scriptsize Error}}}
    \put(-8,135){\rotatebox{90} {{\scriptsize Error}}}
    \put(-8,225){\rotatebox{90} {{\scriptsize Error}}}
    \put(-8,315){\rotatebox{90} {{\scriptsize Error}}}
    \put(40,357){{\scriptsize First Moment}}
    \put(155,357){{\scriptsize Second Moment}}
    \put(273,357){{\scriptsize Third Moment}}
    \put(393,357){{\scriptsize Fourth Moment}}
    \put(60,2){{$\kappa$}}
    \put(178,2){{$\kappa$}}
    \put(297,2){{$\kappa$}}
    \put(417,2){{$\kappa$}}
\end{picture}
\caption{{\small $f(\Delta) := e^{-10\Delta^2}$.}}
\figlabel{gp4}
\end{subfigure}
\caption{\small{Accuracy of PC approximations of different orders ($\kappa$) for various functions using standard GP and constrained $\mathcal{L}_2$ formulations.
Plots show absolute error, $\left|\mathbb{E}\big[f^m(\param)\big] - \mathbb{E}\big[\hat{f}^m(\param)\big]\right|$, for $m=$ 1, 2, 3, and 4, where $\hat{f}$ is the approximated function, and $\Delta$ is uniformly distributed over $[-1,1]$.}}
\figlabel{conGPC}
\end{figure}
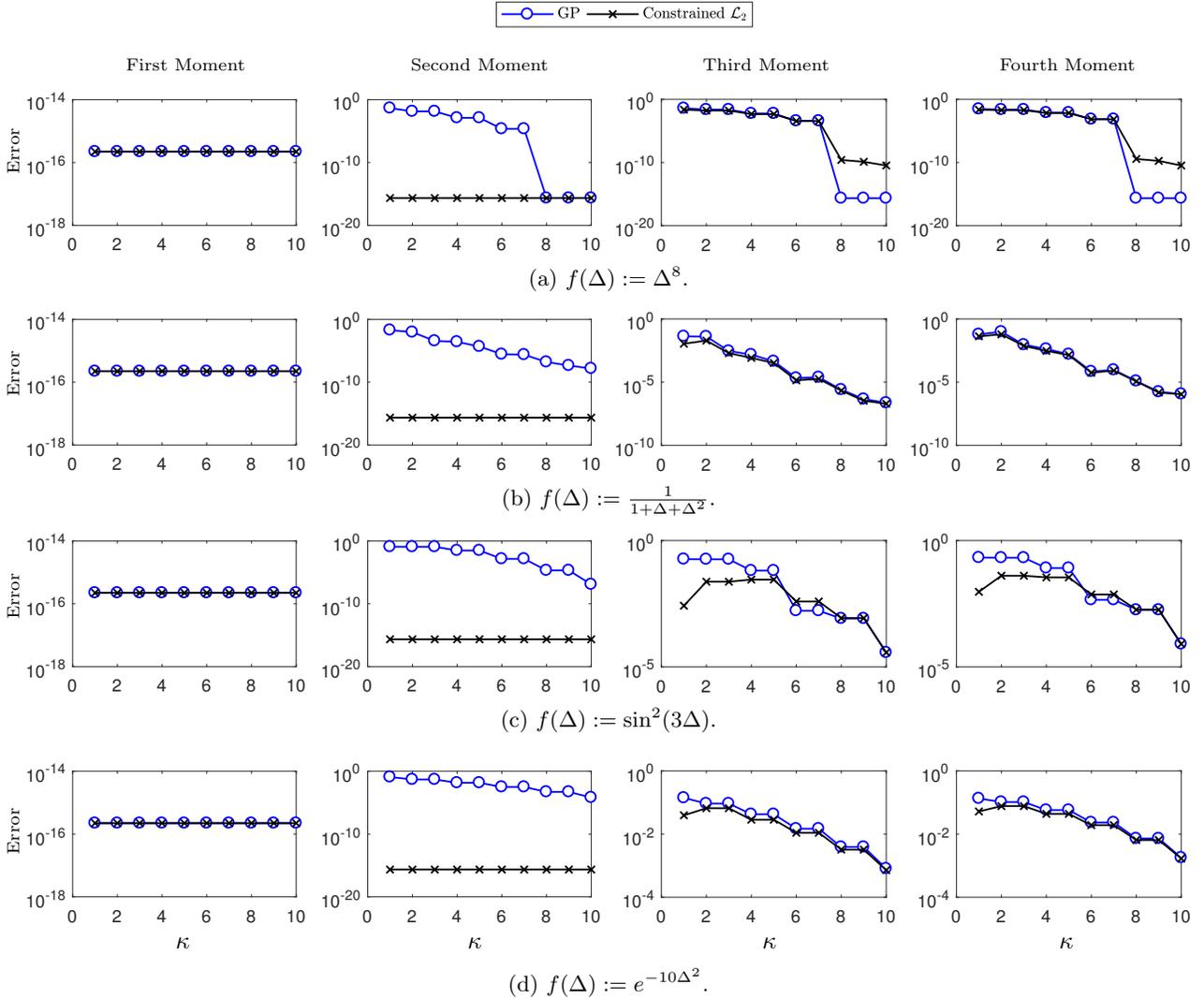

\Fig{conGPC} shows the accuracy of the proposed constrained $\mathcal{L}_2$ method. We compare the errors from standard GP formulation with the errors from the proposed formulation, for increasing approximation order, for a few candidate functions. The errors are plotted on a semi-log scale, lower bounded by the machine precision.
We use Legendre polynomials as basis functions since $\Delta$ is uniformly distributed over $[-1,1]$.

We observe that the errors in the first moment are zero (lower bounded by the machine precision) for both standard GP and constrained $\mathcal{L}_2$ formulations. This is theoretically expected from the two formulations. However, the errors in second moment of the candidate functions are quite large for the standard GP formulation. For polynomials, see \fig{gp1} for $f(\Delta):=\Delta^8$, the error in the second moment from standard GP formulation is quite large, until the approximation order is increased to $8^{\text{th}}$ order. For non polynomial functions, $f(\Delta):=\frac{1}{1+\Delta+\Delta^2}$ in \fig{gp2}, $f(\Delta):=\sin^2(3\Delta)$ in \fig{gp3}, and $f(\Delta):=e^{-10\Delta^2}$ in \fig{gp4}, the errors are quite large even for high orders of approximation. With the constrained $\mathcal{L}_2$ formulation, we observe that the errors in second moment are within machine precision as expected, even for low order approximations. Thus, with the constrained $\mathcal{L}_2$ formulation, we can obtain lower order surrogate models that capture first and second order statistics exactly. This is a significant advantage over the standard GP formulation, which may require very high orders of approximation, to achieve the required statistical accuracy for some functions.

\Fig{conGPC} also shows the errors in $3^\text{rd}$ and $4^\text{th}$ moments. We observe that the errors in these moments for standard GP and constrained $\mathcal{L}_2$ formulations are quite close. Therefore, the proposed formulation does not have significant effect on the accuracy of higher moments, with respect to the standard GP formulation.

\subsection{Constrained $l_2$-optimal approximation}
The $\mathcal{L}_2$ or the GP approach requires integration of $\vo{f}(\param)$, which may be difficult for some functions which are complex. The stochastic collocation (SC) approach overcomes this difficulty by choosing the basis functions $\{\psi_i(\param)\}$ to be multivariate interpolating polynomials, that interpolate over the data $(\param_j,\vo{f}_j)$ defined by
\begin{align}
\vo{f}_j:=\vo{f}(\param_j), \eqnlabel{sc:fi}
\end{align}
where $\{\param_j\}_{j=0}^{N_{\text{SC}}}$ are discrete points in $\domainD$. Typically, the interpolation is achieved using
$$
\fp \approx \vo{\hat{f}}_{\text{SC}}(\param) := \sum_{i=0}^{N_{\text{SC}}} \vo{f}_i \psi_i(\param) =: \F_{\text{SC}}\vo{\Psi}(\param),
$$
where $\vo{\Psi}(\param)$ is a vector of Lagrange polynomials, $\psi_i(\param)$, defined by
\begin{align*}
\psi_i(\param) := \prod_{j=0,j\neq i}^{N_{\text{SC}}} \frac{(\Delta_1-\Delta_{1j})(\Delta_2-\Delta_{2j})\cdots(\Delta_d-\Delta_{dj})}{(\Delta_{1i}-\Delta_{1j})(\Delta_{2i}-\Delta_{2j})\cdots(\Delta_{di}-\Delta_{dj})}, \, i=0,1,\cdots,N_{\text{SC}}.
\end{align*}
Note, $\param_j := [\Delta_{1j}, \Delta_{2j}, \cdots, \Delta_{dj}]^T \in \Real^d$ and $\F_{\text{SC}} := [\vo{f}_0, \vo{f}_1, \cdots, \vo{f}_{N_{\text{SC}}}]$.

Then the moments of $\fp$ are approximated directly by computing the moments of $\vo{\hat{f}}_{\text{SC}}(\param)$. However, there is no guarantee that the moments computed from $\vo{\hat{f}}_{\text{SC}}(\param)$ will match the true moments. Additionally, SC approach can be less accurate than GP formulation for low order approximations, as shown in \fig{pcErrors}.
Therefore, we need to explicitly impose constraints on $\F_{\text{SC}}$, similar to \eqn{gp:quadConstr} and \eqn{gp:partF}, to exactly match the first and second order moments.
However, in the SC formulation, for $(N_{\text{SC}} +1)$ grid points, there is equal number of interpolation polynomials $\psi_i(\Delta)$ of degree $N_{\text{SC}}$. Therefore, there are no extra degrees of freedom to satisfy additional moment constraints. As we can see from \fig{pcErrors}, errors in moments converge to zero when higher order approximations are used.
However, if it is desired to determine a lower order approximation from a given data set, such that first and second order statistics are recovered exactly, moment constraints need to be imposed, which is not admissible in the SC framework.
Therefore, we consider the method of least squares (LS), also known as linear regression or point collocation \cite{walters2003towards,hosder2007efficient}, with moment constraints below.

First, we present the standard LS solution. We begin with selecting a suitable orthogonal polynomial basis $\vo{\Phi}(\param)$ such as defined in \eqn{Phi:def}. Then, $\fhat(\param) = \F\mo{\Phi}(\param)$.
A grid $G_{l_{2}}:=\{\param_i\}_{i=1}^{N_p}\in\domainD$ is created by generating $N_p$ samples from the random space $\domainD$. The LS solution for $\F$ is determined such that the sum of squares of error norms at each grid point is minimized, i.e.,
\begin{align*}
\min_{\F}\sum_{i=1}^{N_p}\|\F\vo{\Phi}(\param_i) - \vo{f}(\param_i)\|^2.
\end{align*}
The optimal $\F$ for this optimization problem is denoted as $\F_{\text{LS}}$, and is given by
\begin{align*}
\F_{\text{LS}} &:= \H_1^T\H_2^{-1}
\end{align*}
where,
\begin{align*}
\H_1 :=  \sum_{i=1}^{N_p}\vo{\Phi}(\param_i)\vo{f}(\param_i)^T, \text{ and } \H_2 := \sum_{i=1}^{N_p}\vo{\Phi}(\param_i)\vo{\Phi}^T(\param_i).
\end{align*}
Note, the existence of such $\F_{\text{LS}}$ requires that $N_p \geq (N+1)$. For the purpose of numerical experiments presented in this paper, we have used $N_p = 2(N+1)$, as recommended in \cite{hosder2007efficient}.

We next present an $l_2$-optimal approximation of $\fp$ with moment constraints. The moment constraints on $\F$ are imposed as (see \eqn{gp:meanFhat}, \eqn{gp:covFhat}, \eqn{gp:quadConstr} and \eqn{gp:partF})
\begin{align}
\F\begin{bmatrix}1 & 0 & \cdots & 0\end{bmatrix}^T = \E{\fp},  \eqnlabel{sc:mean} \\
 \F\W\F^T =  \E{\fp\fp^T}. \eqnlabel{sc:cov}
\end{align}
Unlike SC formulation, the polynomial coefficients or the columns of $\F$ won't be the function values at $\param_j$ but the solution of the optimization problem presented next.
Note that in the constrained $\mathcal{L}_2$ formulation, the true moments are computed by integration of $\fp$ over the parameter domain $\domainD$. Here, in $l_2$ formulation, the true moments can be computed by evaluating $\fp$ over a finite grid $G:=\{\param_j\}\in\domainD$. E.g., $\E{\fp} =  \sum_j w_j \vo{f}(\param_j)$ where $w_j$ is the weight associated with the grid point $\param_j$.

The $l_2$-optimal approximation of $\fp$ over the grid $G_{l_{2}}:=\{\param_i\}_{i=1}^{N_p}\in\domainD$ that recovers the first and second order statistics of $\fp$ is obtained from the following optimization problem,
\begin{align*}
\min_{\F}\sum_{i=1}^{N_p}\|\F\vo{\Phi}(\param_i) - \vo{f}(\param_i)\|^2, \text{ subject to \eqn{sc:mean} and \eqn{sc:cov}. }
\end{align*}
Therefore, this is a constrained least squares formulation to approximate $\fp$.
Note, the grid $G_{l_{2}}$ need not be the same as the grid $G$ that is used to compute the true moments. In fact, $G_{l_{2}}$ can be sparser than $G$.
This optimization problem is nonconvex due to the constraint in \eqn{sc:cov}. We use the same approach as in the $\mathcal{L}_2$ formulation, to solve the problem. The optimization problem is formulated using unitary matrices, solved without the nonconvex constraint, and then the optimal solution is projected on the nonconvex set. The $l_2$-optimal  solution with constraints on first and second moments is given by the following theorem.

\begin{theorem}
The coefficients that result in the $l_2$-optimal PC approximation, subject to constraints on first and second-order statistics, is given by
\begin{align}
\F := \begin{bmatrix} \E{\fp} &  \vo{L}\vo{U}^\ast\W_1^{-1/2} \end{bmatrix},
\end{align}
where $\vo{U}^\ast :=  \vo{M}_1\vo{T}\vo{M}^T_2$, $\vo{M}_1$ and $\vo{M}_2$ are  unitary matrices obtained from the singular value decomposition of $\hat{\vo{U}}$ that is defined in \eqn{sc:U1}, and  $\vo{T}:=\begin{bmatrix}\I{n} & \vo{0}_{n\times(N-n)}\end{bmatrix}$.
\label{thm:bestl2}
\end{theorem}

\begin{proof}
We define $\vo{e}(\param) := \vo{f}(\param) - \F\vo{\Phi}(\param)$, and seek an optimal $\F$ that minimizes $\trace{\sum_{i=1}^{N_p} \vo{e}(\param_i)\vo{e}^T(\param_i)}$, subject to \eqn{sc:mean} and \eqn{sc:cov}.
Using the results of Theorem \ref{thm:gp1}, constraint \eqn{sc:mean}  is satisfied when $\F$ is partitioned as $\F = \begin{bmatrix}\E{\fp} & \F_1\end{bmatrix}$, and \eqn{sc:cov} is satisfied when $\F_1 = \L\vo{U}\W_1^{-1/2}$, for any $\vo{U}$ satisfying $\vo{U}\vo{U}^T=\I{n}$.
Now, let us write the cost function as a function of $\F_1$
\begin{align}
J(\F_1) := \trace{\sum_{i=1}^{N_p} \vo{e}(\param_i)\vo{e}^T(\param_i)} & = \trace{\EE_0 - \EE_1^T\F_1^T - \F_1\EE_1 + \F_1\EE_2\F_1^T},
\eqnlabel{JF1}
\end{align}
where
\begin{align*}
\EE_0 & := \sum_{i=1}^{N_p}\big(\vo{f}(\param_i) - \E{\fp}\big)\big(\vo{f}(\param_i) - \E{\fp}\big)^T, \\
\EE_1 & := \sum_{i=1}^{N_p}\vo{\Phi}_{1}(\param_i)\big(\vo{f}(\param_i) - \E{\fp}\big)^T, \text{ and } \\
\EE_2 &:= \sum_{i=1}^{N_p}\vo{\Phi}_{1}(\param_i)\vo{\Phi}^T_{1}(\param_i).
\end{align*}
Therefore, $J(\F_1)$ is minimized when
$$
\frac{\partial}{\partial \F_1} \trace{\EE_0 - \EE_1^T\F_1^T - \F_1\EE_1 + \F_1\EE_2\F_1^T} = 0,
$$
or
\begin{equation}
\F_1 = \EE_1^T\EE_2^{-1}. \eqnlabel{sc:F1}
\end{equation}
Equating the solution from \eqn{sc:F1} with the $\F_1$ that satisfies \eqn{sc:cov}, we get
\begin{align}
& \L\vo{U}\W_1^{-1/2} = \EE_1^T\EE_2^{-1} \\
\text{or } \quad & \hat{\vo{U}} := \L^{-1}\EE_1^T\EE_2^{-1}\W_1^{1/2}. \eqnlabel{sc:U1}
\end{align}
However, $\hat{\vo{U}}$ may not necessarily satisfy the constraint $\vo{U}\vo{U}^T = \I{n}$ that is necessary for matching the second moment.
Therefore, similar to the projection approach that we use in \Cref{thm:bestL2}, let the singular-value decomposition of $\hat{\vo{U}}$ be $\vo{M}_1\vo{D}\vo{M}^T_2$, i.e.,
\begin{align}
\vo{M}_1\vo{D}\vo{M}_2^T = \L^{-1}\EE_1^T\EE_2^{-1}\W_1^{1/2},
\eqnlabel{sc:svd}
\end{align}
where $\vo{D}=\begin{bmatrix}\vo{\Lambda} & \vo{0}_{n\times(N-n)}\end{bmatrix}$, and $\vo{\Lambda}$ is diagonal matrix with singular values of $\hat{\vo{U}}$. Defining $\vo{U}^\ast := \vo{M}_1\vo{T}\vo{M}^T_2$  with $\vo{T}:=\begin{bmatrix}\I{n} & \vo{0}_{n\times(N-n)}\end{bmatrix}$,
ensures $\vo{U}^\ast\vo{U}^{\ast T} = \I{n}$, and gives us the $l_2$-optimal solution.
\end{proof}

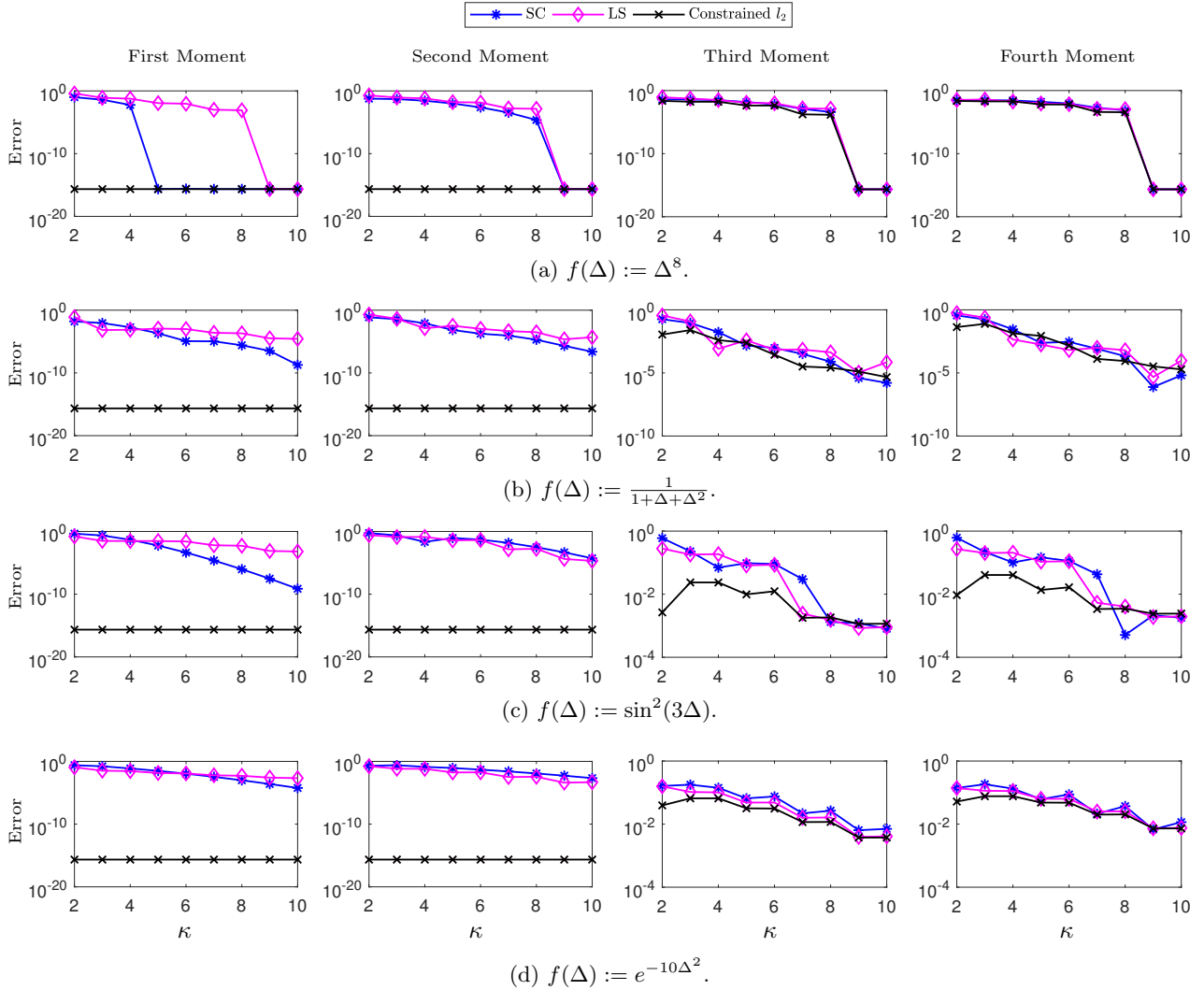
\begin{figure}[ht!]
\begin{subfigure}[b]{\textwidth}
\includegraphics[width=\textwidth]{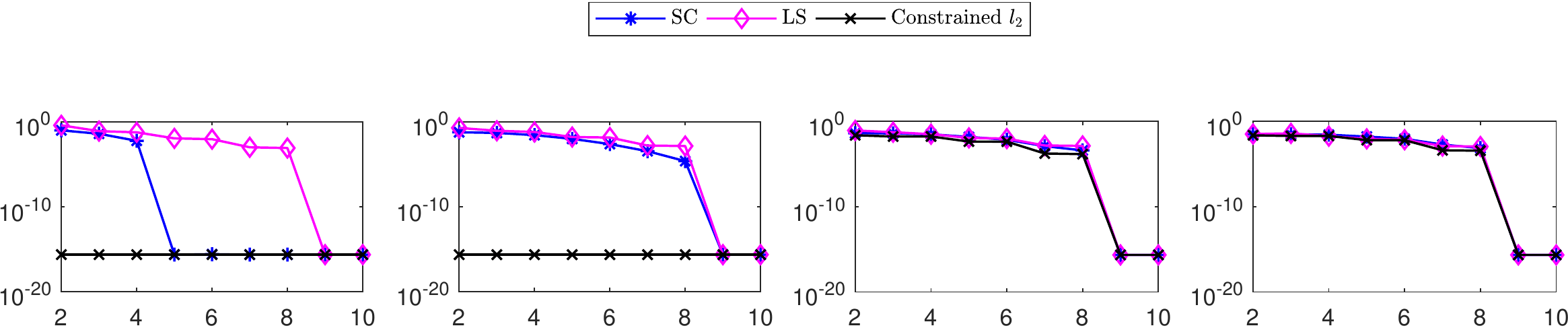}
\caption{{\small $f(\Delta) := \Delta^8$.}}
\figlabel{sc1}
\end{subfigure}\vspace{2mm}
\begin{subfigure}[b]{\textwidth}
\includegraphics[width=\textwidth]{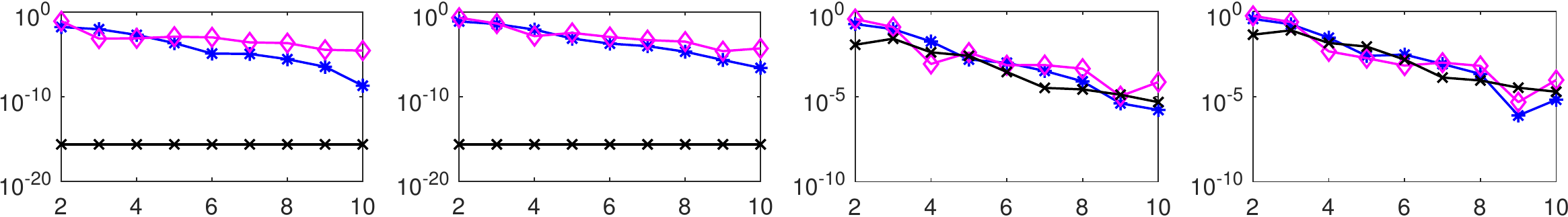}
\caption{{\small $f(\Delta) := \frac{1}{1+\Delta+\Delta^2}$.}}
\figlabel{sc2}
\end{subfigure}\vspace{2mm}
\begin{subfigure}[b]{\textwidth}
\includegraphics[width=\textwidth]{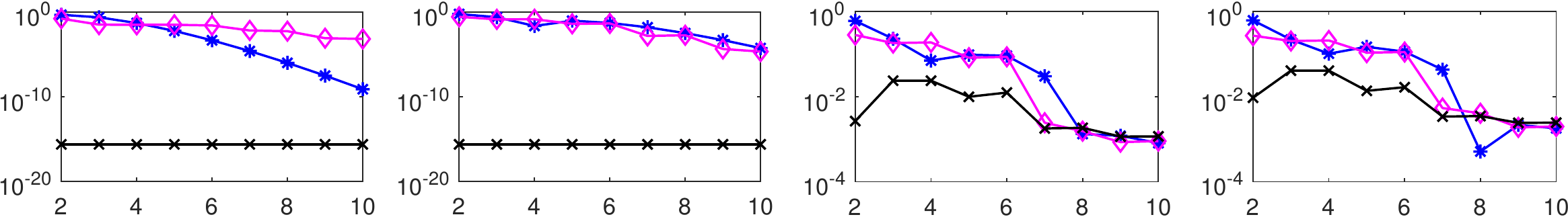}
\caption{{\small $f(\Delta) := \sin^2(3\Delta)$.}}
\figlabel{sc3}\vspace{1em}
\end{subfigure}
\begin{subfigure}[b]{\textwidth}
\includegraphics[width=\textwidth]{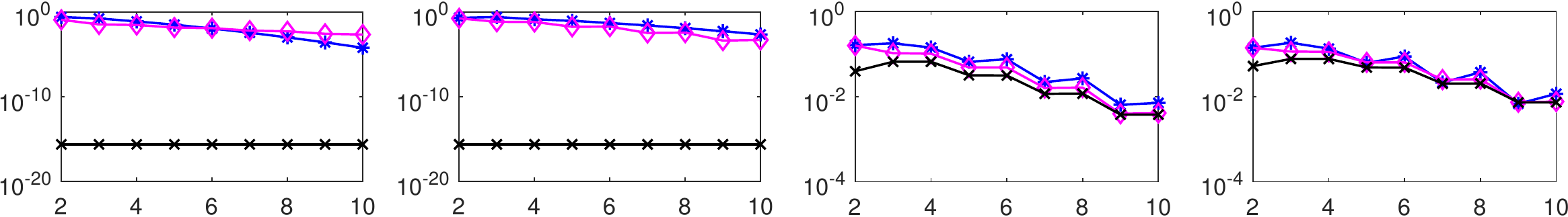}
\begin{picture}(0,0)
    \put(-8,40){\rotatebox{90} {{\scriptsize Error}}}
    \put(-8,135){\rotatebox{90} {{\scriptsize Error}}}
    \put(-8,225){\rotatebox{90} {{\scriptsize Error}}}
    \put(-8,315){\rotatebox{90} {{\scriptsize Error}}}
    \put(40,357){{\scriptsize First Moment}}
    \put(155,357){{\scriptsize Second Moment}}
    \put(273,357){{\scriptsize Third Moment}}
    \put(393,357){{\scriptsize Fourth Moment}}
    \put(60,2){{$\kappa$}}
    \put(178,2){{$\kappa$}}
    \put(297,2){{$\kappa$}}
    \put(417,2){{$\kappa$}}
\end{picture}
\caption{{\small $f(\Delta) := e^{-10\Delta^2}$.}}
\figlabel{sc4}
\end{subfigure}
\caption{{\small Accuracy of PC approximations of different orders ($\kappa$) for various functions using standard SC, LS, and constrained $l_2$ formulations.
Plots show absolute errors $\left|\mathbb{E}\big[f^m(\param)\big] - \mathbb{E}\big[\hat{f}^m(\param)\big]\right|$, for $m=1$, $2$, $3$, and $4$ respectively, where $\hat{f}$ is the approximated function, and $\Delta$ is uniformly distributed over $[-1,1]$.}}
\figlabel{conSC}
\end{figure}

In \fig{conSC}, we compare the errors in moments determined using proposed constrained $l_2$ formulation with the stochastic collocation (SC) method, and the least squares (LS) solution.
We observe that the errors in the first moment are zero for the constrained $l_2$ formulation, for all orders of approximation. This is expected, as the constraint on the first moment was explicitly imposed.
On the other hand, errors in the first moment determined using SC and LS methods are quiet large for low order approximations. For polynomial function in \fig{sc1}, the error in first moment becomes within machine precision for $5^\text{th}$ and $9^\text{th}$ order approximations corresponding to SC and LS methods respectively.
Overall, the LS method has the poorest accuracy in estimating the first moment among three methods under discussion.

The errors in second moment of the candidate functions due to the least squares and SC formulations are similar. As observed from \fig{sc1}, for the polynomial function, the error in second moment becomes within machine precision for the $9^{\text{th}}$ order approximation. For other candidate functions, the error due to these two methods is large even for higher orders of approximation.
We observe that the errors in second moments of all candidate functions are within machine precision for the constrained $l_2$ approach, even for low order approximations.
Thus, similar to the constrained $\mathcal{L}_2$ formulation, the constrained $l_2$ formulation also can be used to obtain lower order surrogate models that capture first and second order statistics exactly.
We also observe that the errors in $3^\text{rd}$ and $4^\text{th}$ moments for the constrained $l_2$ formulation are comparable with the least squares and SC methods, as we imposed the constraints that match only first two moments exactly.

In this section, we have presented constrained $\mathcal{L}_2$ and constrained $l_2$ formulations that can recover the first two moments exactly. We use these results to develop piecewise linear surrogate models for time series data and stochastic ODEs, which is presented next.

\section{Linear propagator with constrained mean and variance} \label{sec:lfp}
In this section, we develop a piecewise linear model for temporal evolution of PC coefficients for dynamic systems. The formulation uses time series data which can be obtained from real experiments or computer simulations, to fit a least squares or $l_2$-optimal linear model. As discussed below, the surrogate model can also be used as a linear propagator for  systems with known nonlinear dynamics, which are often represented by ordinary differential equations. Data in this case is based on \textit{observed} or \textit{measured} state variables. For stochastic ODEs, one often derives a deterministic surrogate propagator based on GP approach.


We begin with a brief derivation of the standard GP-based propagator. Let us consider the following stochastic ordinary differential equation
\begin{align}
\xdot(t,\param) = \vo{f}(\x(t,\param)), \eqnlabel{genODE}
\end{align}
where, $\x:=\x(t,\param) \in\real^n$ is the solution for given initial condition $\x_0(\param)$, and $\vo{f} : \real^n \mapsto \real^n$ can be a non-linear function. Therefore, an approximate solution similar to \eqn{pcF} is given by
$$
\x(t,\param) \approx \vo{\hat{x}}(t,\param) := \sum_{i=0}^N\x_i(t)\phi_i(\param),
$$
or
\begin{align}
\vo{\hat{x}}(t,\param) = \X(t)\vo{\Phi}(\param), \eqnlabel{XNODE}
\end{align}
where $\X(t): \real \mapsto \real^{n\times(N+1)}$, and $\vo{\Phi}(\param)$ is defined in \eqn{Phi:def}. Substituting the approximate solution from \eqn{XNODE} in \eqn{genODE}, we get the residue
\begin{align}
\vo{e}(t,\param) := \vo{\dot{X}}\vo{\Phi}(\param) - \vo{f}(\X(t)\vo{\Phi}(\param)). \eqnlabel{res1}
\end{align}
Let us define $\xpc:=\vec{\X}$. In the standard GP approach, the projection of $\vo{e}(t,\param)$ on $\phi_i(\param)$ is set to zero, i.e. $\E{\vo{e}(t,\param)\phi_i(\param)} = 0$ for $i=0,\cdots,N$. This results in the $(N+1)$ deterministic ordinary differential equations in the elements of $\xpc$. For the clarity of discussion, let us denote $\xpc$ determined using the GP approach by $\xgpc$. Then the system of coupled deterministic ODEs can be represented by
\begin{align}
\dot{\x}_{\text{gpc}}(t) = \vo{g}(\xgpc(t)). \eqnlabel{genXpcODE}
\end{align}
and temporal discretization can be carried out using standard methods such as Runge-Kutta schemes to solve \eqn{genXpcODE} numerically. Therefore, a discretized system can be written as
\begin{align}
\xgpc^{k+1} = \vo{G}(\xgpc^{k}),
\eqnlabel{dsicXpcRK}
\end{align}
and is often referred as \textit{surrogate propagator} for the original system. Thus, by propagating the PC coefficients forward in time using deterministic system in \eqn{dsicXpcRK}, we can find the approximate stochastic solution, $\vo{\hat{x}}(t,\param)$, in \eqn{XNODE}.
It is desired that the equation error is unbiased, i.e., $\E{\vo{e}(t,\param)}=0$, and has the minimum variance. In \cref{sec:gpUnbiased}, we show that the standard GP approach leads to such kind of $\mathcal{L}_2$ optimal PC approximation.
However, as shown in \fig{1D_mean}-\ref{fig:nonLin_variance} and discussed later, the unbiased and minimum variance nature of the equation error does not guarantee the accuracy of statistical moments of the approximate solution $\vo{\hat{x}}(t,\param)$. In fact, statistical moments of the approximate solution deviate appreciably from the true moments as we propagate the system forward in time.
However, they are quite accurate for the initial small time interval, and we employ such GP based propagator to initialize the propagator proposed in this section.

Next, we present formulation for the proposed propagator that uses temporal sequence of states to approximate a piecewise linear model for the PC coefficients.
The principal assumption for this approach is that at each time step $k$, the true or reference mean, covariance, and the matrix $\R$ defined in \eqn{R}, can be calculated for the state vector.
The objective here is to use the information available up to time step $k$ to predict the moments of state vector at the next time step $k+1$.
Owing to the assumption, we can directly use the results from \Cref{thm:bestL2} or \ref{thm:bestl2} to determine the optimal PC coefficients using constrained formulations that recover the first two moments at any given time step $k$. To distinguish from the GP based coefficients, we denote the PC coefficients determined using constrained formulations by $\xcpc$.
Next, we determine a constant, $l_2$-optimal state transition matrix, $\vo{M}$, for a given finite sequence of optimal PC coefficients $\{\xcpc^j\}_{j=k-q}^{k}$, i.e., we seek an optimal $\vo{M}$ that satisfies
\begin{align}
\xcpc^{j+1} = \vo{M}\xcpc^{j} \text{ for }  j= k-q, \cdots k-1.
\eqnlabel{xpcDynLinFit}
\end{align}
such that the equation error of \eqn{xpcDynLinFit} is minimized in the least squares or $l_2$-sense.
A general result for this optimization problem is presented as the following theorem.
\begin{theorem}
  Given a sequence of state vector $\x$, $\{\x^j\}_{j=0}^{q} \subset \Real^n$ for a positive integer $q$, the optimal state transition matrix, $\vo{M}^*$, which approximates the sequence using a linear model that minimizes the error in $l_2$-sense is given by
  \begin{align}
    \vo{M}^* := \Big[\sum_{j=0}^{q-1} \x^{j+1}(\x^{j})^T \Big] \Big[\sum_{j=0}^{q-1} \x^{j}(\x^{j})^T \Big]^{-1},
    \eqnlabel{defoptM}
  \end{align}
  subject to the existence of the inverse involved.
  \label{thm:optM}
\end{theorem}
\begin{proof}
We seek an $\vo{M} \in \Real^{n\times n}$ that minimizes the error of equation  $\x^{j+1} = \vo{M}\x^j$ in the least squares or $l_2$-sense for $j=0,1,\cdots,q-1$. Therefore, we define the cost as a function of $\vo{M}$ as
$$
J(\vo{M}):=\sum_{j=0}^{q-1} || \x^{j+1} - \vo{M}\x^{j} ||^2.
$$
The cost function can be written as
\begin{align*}
J(\vo{M}) &= \sum_{j=0}^{q-1} (\x^{j+1} - \vo{M}\x^{j})^T(\x^{j+1} - \vo{M}\x^{j}) \\
&= \sum_{j=1}^{q} (\x^{j})^T\x^{j} + \sum_{j=0}^{q-1} (\x^{j})^T\vo{M}^T\vo{M}\x^{j} - 2\sum_{j=0}^{q-1} (\x^{j+1})^T\vo{M}\x^{j}
\end{align*}
Taking derivative of $J(\vo{M})$ w.r.t. $\vo{M}$ and equating it to zero provides us the optimal $\vo{M}^*$ that achieves the minimum of the cost function, i.e.,
\begin{align*}
\frac{\partial J(\vo{M})}{\partial \vo{M}} &=
2\vo{M} \Big[\sum_{j=0}^{q-1} \x^{j}(\x^{j})^T \Big] - 2\Big[\sum_{j=0}^{q-1} \x^{j+1}(\x^{j})^T \Big] = 0, \\
\implies  \vo{M}^* &= \Big[\sum_{j=0}^{q-1} \x^{j+1}(\x^{j})^T \Big] \Big[\sum_{j=0}^{q-1} \x^{j}(\x^{j})^T \Big]^{-1},
\end{align*}
provided the inverse exists.
\end{proof}
 Note that $\vo{M}$ which satisfies \eqn{xpcDynLinFit} approximates the PC coefficients from time step $k-q$ to $k$ using an $l_2$-optimal linear model and we denote it by $\vo{M}^k$.
Using the result from \cref{thm:optM}, the $l_2$-optimal $\vo{M}^k$ can be written as
\begin{align}
  \vo{M}^k := \Big[\sum_{j=0}^{q-1} \xcpc^{k-j}(\xcpc^{k-j-1})^T \Big] \Big[\sum_{j=1}^{q} \xcpc^{k-j}(\xcpc^{k-j})^T \Big]^{-1}.
  \eqnlabel{defMk}
\end{align}
We use this $\vo{M}^k$ to estimate the optimal PC coefficients at the next time step $k+1$ as
\begin{align}
\hat{\x}^{k+1}_{\text{pc}} = \vo{M}^k \xcpc^k.
\eqnlabel{xpcNext}
\end{align}
Note, each matrix $\xcpc^{k-j}(\xcpc^{k-j})^T$, $j=1,2,\cdots q$ in \eqn{defMk} is a rank one matrix for non-zero $\xcpc^{k-j}$. Therefore, the inverse exists only if $q \geq n(N+1)$, i.e., $q$ must be at least equal to the number of elements in $\xcpc$.
Therefore, we can propagate and predict PC coefficients only for time steps $k \geq n(N+1)$. For $k<n(N+1)$, we use the solution obtained by solving deterministic ODEs derived from GP approach given by \eqn{dsicXpcRK} as it is quiet accurate for the initial small time interval. We summarize the proposed formulation as an algorithm, outlined in \algo{lfpAlgo}.

\begin{algorithm}[H]
\SetAlgoLined
 Initialize $\xpc^0$\\
 Choose $q \geq n(N+1)$\;
 \For{$k=0,1,2,\cdots$}{
 determine $\xcpc^k$ using reference moments and Theorem \ref{thm:bestL2} or \ref{thm:bestl2}\;
  \eIf{$k<q$}{
   Use \eqn{dsicXpcRK} to estimate $\xpc^{k+1}$\;
   }{
   Determine the $l_2$-optimal $\vo{M}^k$ using \eqn{defMk}\;
   Use \eqn{xpcNext} to estimate $\xpc^{k+1}$\;
  }
  }
 \caption{Linear propagator}
 \label{lfpAlgo}
\end{algorithm}

Numerical results obtained using this algorithm are presented in the following section.

\section{Numerical results} \label{sec:num_results}
In this section, we consider stochastic ODEs and solve them using conventional GP based propagator and \algo{lfpAlgo}, and compare the results. As shown in \cref{sec:gpUnbiased}, the GP method leads to $\mathcal{L}_2$-optimal PC approximation with unbiased equation error that has minimum variance.
However, this doesn't imply accuracy in the  moments of the solution.

In addition, due to non-zero equation error, for finite $N$, the GP framework is only well suited for evaluating short term statistics (i.e. for $0 \leq t < T$ for some $T$). It is well known that the accuracy of the estimates degrades over time  for $t\geq T$. The extent of this deviation depends on the variability in the state trajectories due to the parametric uncertainty. However, for stable systems the accuracy may improve over time as the variability diminishes when $t\rightarrow\infty$, because all the sample paths reach the origin.
Therefore, statistics obtained from any finite dimensional gPC approximation of a stable process will also converge to zero, thus possibly, only matching the true statistics initially and asymptotically.
To highlight this degradation, we begin with the following linear stochastic ODE.

\subsection{Linear ODE} \label{linODE}
Let us consider a scalar system $\dot{x} = -ax$, where $a$ is uniformly distributed over $[0,1]$. For such a system, the analytical expression for mean and variance is given by $\bar{x}(t) = \frac{1-e^{-t}}{t}$ and $\sigma(t) = \frac{1-e^{-2t}}{2t}-\left(\frac{1-e^{-t}}{t}\right)^2$ for initial condition $x_0 = 1$.
We discretize \eqn{gpcLinSurr} using $4^{\text{th}}$ order Runge-Kutta (RK4) scheme to calculate the moments of $x$ using GP approach. Same equation is used in \algo{lfpAlgo} for $k<q=n(N+1)$.

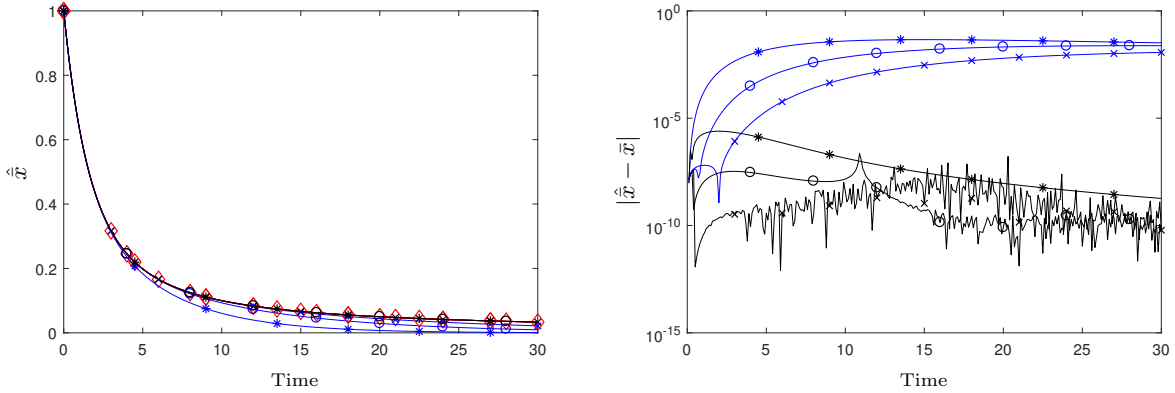
\begin{figure}[ht!]
\begin{center}
        \includegraphics[trim={2cm 0.1cm 2cm 0.2cm},clip,width=\textwidth]{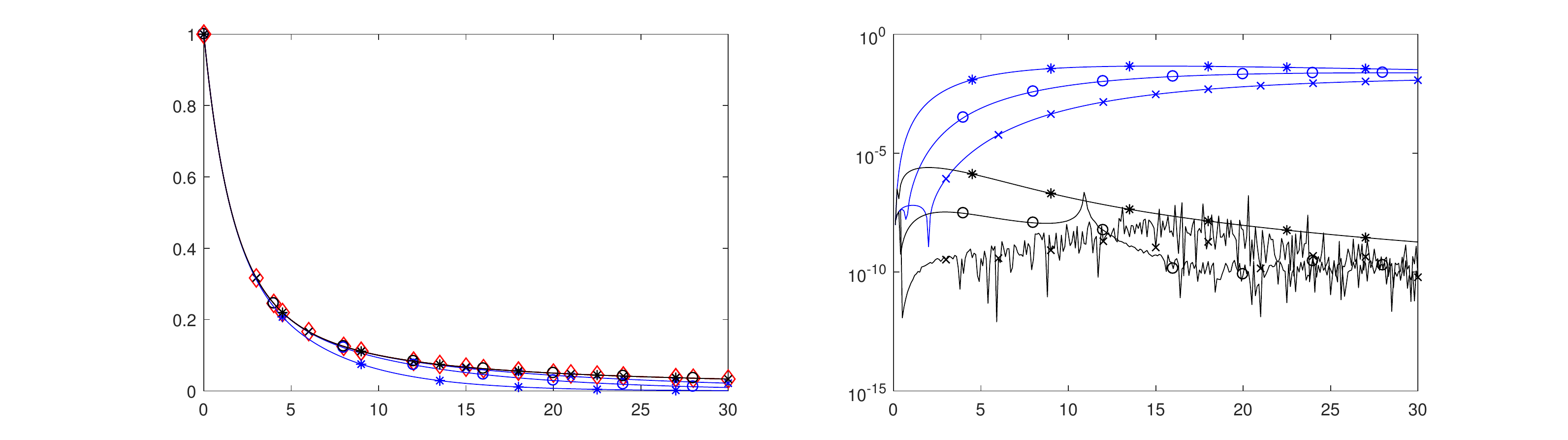}
\begin{picture}(0,0)
    \put(-220,85){\rotatebox{90} {{\small$\hat{\bar{x}}$}}}
    \put(10,75){\rotatebox{90} {{\small$|\hat{\bar{x}} - \bar{x}|$}}}
    \put(-120,7){{\scriptsize Time}}
    \put(118,7){{\scriptsize Time}}
\end{picture}
\caption{{\small Linear system - Error in mean estimated using GP formulation and \algo{lfpAlgo} for different approximation order $\kappa$. Analytical: red, GP formulation: blue, \algo{lfpAlgo}: black. $\kappa=$ 1(*), 2(o), 3(x).}}
\figlabel{1D_mean}
\end{center}
\end{figure}

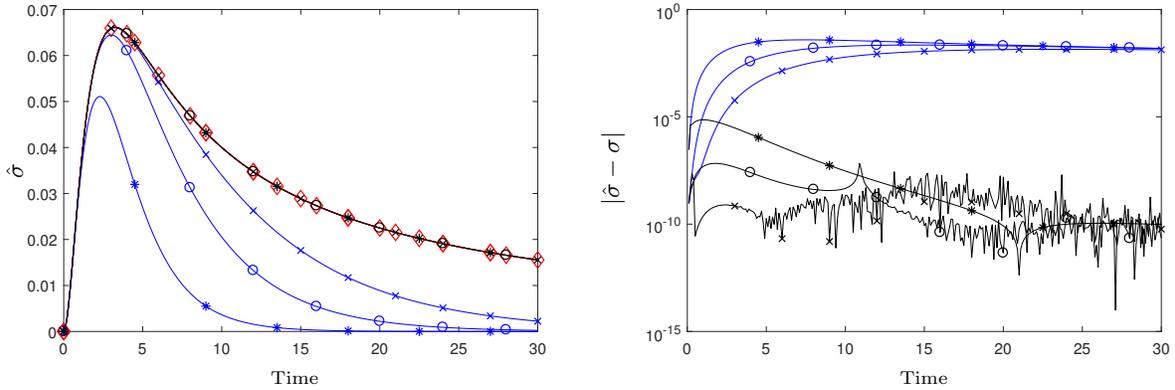
\begin{figure}[ht!]
\begin{center}
        \includegraphics[trim={2cm 0.1cm 2cm 0.2cm},clip,width=\textwidth]{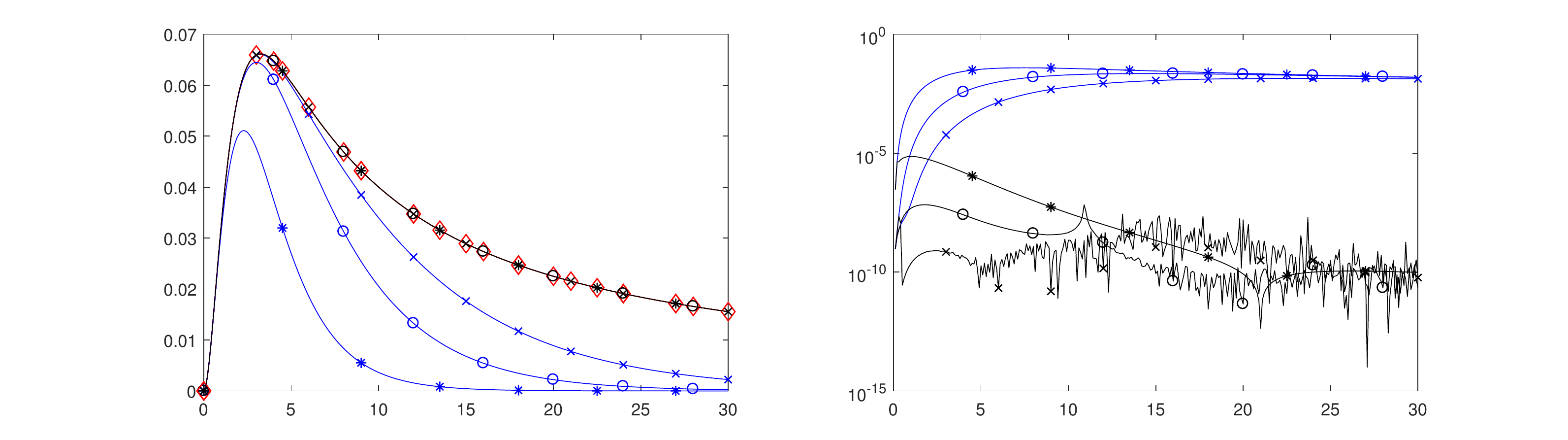}
\begin{picture}(0,0)
    \put(-220,85){\rotatebox{90} {{\small $\hat{\sigma}$}}}
    \put(5,75){\rotatebox{90} {{\small $|\hat{\sigma} - \sigma|$}}}
    \put(-120,7){{\scriptsize Time}}
    \put(118,7){{\scriptsize Time}}
\end{picture}
\caption{{\small Linear system - Error in variance estimated using GP formulation and \algo{lfpAlgo} for different approximation order  $\kappa$. Analytical: red, GP formulation: blue, \algo{lfpAlgo}: black. $\kappa=$ 1(*), 2(o), 3(x).}}
\figlabel{1D_variance}
\end{center}
\end{figure}

\Fig{1D_mean} and \ref{fig:1D_variance} show the errors in mean and variance of $x$ determined using GP approach (blue) and \algo{lfpAlgo} (black) for different orders of approximation, $\kappa$. For the purpose of simulation, we use the known analytical solution to compute $\xcpc^k$.
As expected, it can be observed from \fig{1D_mean} and \ref{fig:1D_variance} that the error in estimated moments decreases with increasing order of approximation $\kappa$. However, we notice that the errors in moments estimated using GP formulation grow with time for all orders of approximation. The time invariant nature of the system matrix in \eqn{gpcLinSurr} causes the deviation of estimated states moments from the analytical solution over time.
On the other hand, the errors in moments estimated using \algo{lfpAlgo} (black) are significantly lower than GP formulation. This is due to the fact that the state transition matrix $\vo{M}^k$ is calculated at each time step by fitting a linear model over reference data at previous time steps to estimate PC coefficients $\hat{\x}_\text{pc}^{k+1}$.

\subsection{Non-linear ODE}
To demonstrate the application of \algo{lfpAlgo} for a non-linear stochastic systems, let us consider the following candidate ODE
\begin{align}
\dot{x} = -ax^2 + \sin{x},
\eqnlabel{nonLinODE}
\end{align}
where $a$ is uniformly distributed over $[0,1]$ with initial condition $x_0 = 1$. The reference solution for this system is determined using a Monte-Carlo (MC) simulation, which is also used to compute $\xcpc^k$ in \algo{lfpAlgo}.

After employing GP formulation, \eqn{nonLinODE} is transformed into the following deterministic ODE, see \cref{sec:gpUnbiased},
\begin{align}
\dot{\x}_{\text{gpc}}(t) = \W^{-1}\E{\vo{\Phi}(\param) \otimes \Big(-a((\phipt\otimes 1)\xgpc)^2 + \sin\big((\phipt\otimes 1)\xgpc\big)\Big)},
\eqnlabel{gpxpcNonLinOde}
\end{align}
where $\W$ is defined in \eqn{gp:W}. We use MC method to compute the expectation on the right side of \eqn{gpxpcNonLinOde}, and RK4 to discretize the equation temporally.

\begin{figure}[ht!]
\begin{center}
        \includegraphics[trim={2cm 0.1cm 2cm 0.2cm},clip,width=\textwidth]{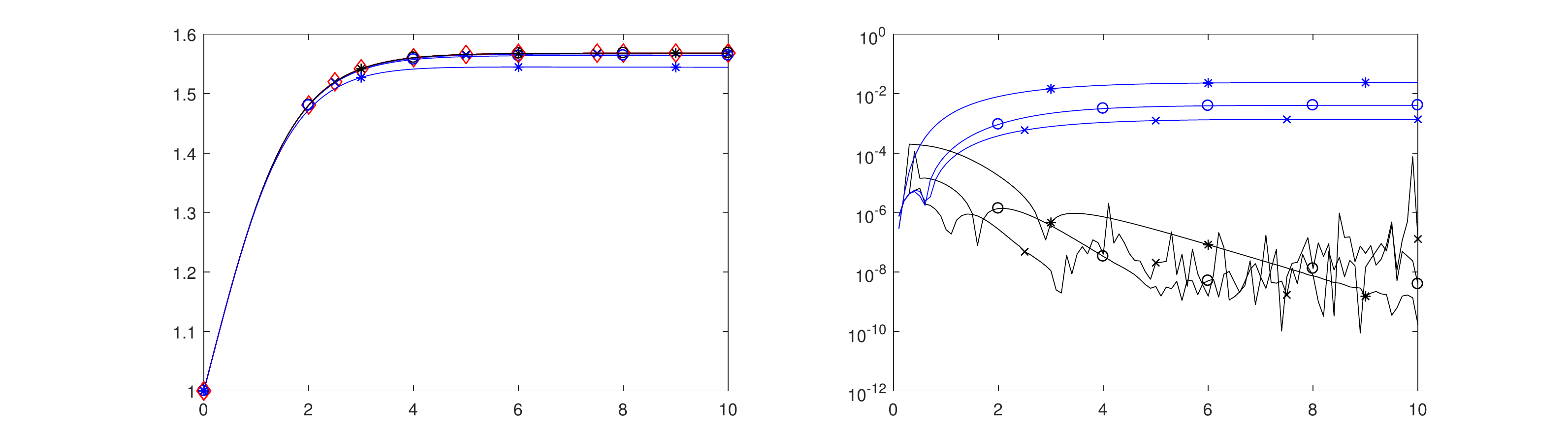}
\begin{picture}(0,0)
    \put(-220,85){\rotatebox{90} {{\small$\hat{\bar{x}}$}}}
    \put(10,75){\rotatebox{90} {{\small$|\hat{\bar{x}} - \bar{x}|$}}}
    \put(-120,7){{\scriptsize Time}}
    \put(118,7){{\scriptsize Time}}
\end{picture}
\caption{{\small Non-linear system - Error in mean estimated using GP formulation and \algo{lfpAlgo} for different $\kappa$. MC: red, GP formulation: blue, \algo{lfpAlgo}: black. $\kappa=$ 1(*), 2(o), 3(x).}}
\figlabel{nonLin_mean}
\end{center}
\end{figure}

\begin{figure}[ht!]
\begin{center}
        \includegraphics[trim={2cm 0.1cm 2cm 0.2cm},clip,width=\textwidth]{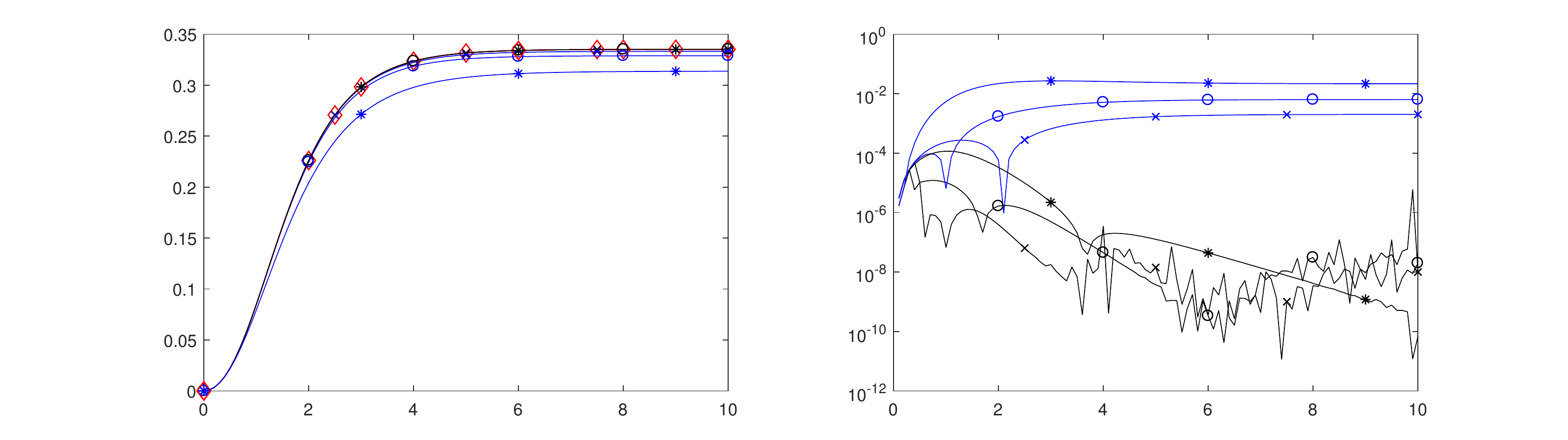}
\begin{picture}(0,0)
    \put(-220,85){\rotatebox{90} {{\small $\hat{\sigma}$}}}
    \put(5,75){\rotatebox{90} {{\small $|\hat{\sigma} - \sigma|$}}}
    \put(-120,7){{\scriptsize Time}}
    \put(118,7){{\scriptsize Time}}
\end{picture}
\caption{{\small Non-linear system - Error in variance estimated using GP formulation and \algo{lfpAlgo} for different $\kappa$. MC: red, GP formulation: blue, \algo{lfpAlgo}: black. $\kappa=$ 1(*), 2(o), 3(x).}}
\figlabel{nonLin_variance}
\end{center}
\end{figure}
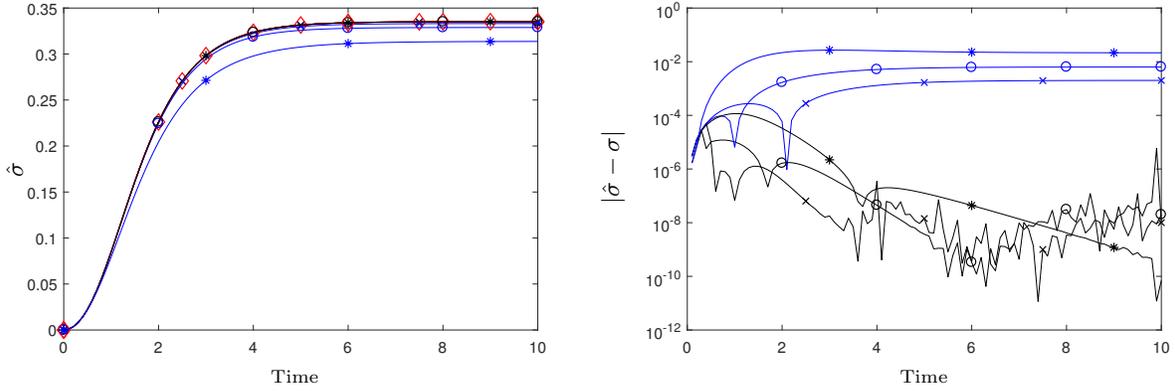

\Fig{nonLin_mean} and \fig{nonLin_variance} compare the mean and variance determined using GP formulation (blue) and \algo{lfpAlgo} (black). Trends observed here are similar to ones observed for linear case.
Once again, the errors for \algo{lfpAlgo} are lower than those for GP formulation.

\begin{remark}
The matrix $\vo{M}^k$ approximates a linear model for $\xcpc$ from time step $k-q$ to $k$. Since $\vo{M}^k \in \Real^{n(N+1) \times n(N+1)}$ and $\xcpc^k \in \Real^{n(N+1)}$,
for $q = n(N+1)$ the degrees of freedom in determining $\vo{M}^k$ match exactly with the data points $\xcpc^k, \xcpc^{k-1}, \cdots ,\xcpc^{k-q}$, and we get a linear approximation, theoretically, with no error. For $q>n(N+1)$, there are more data points than degrees of freedom and it results in a least square solution for $\vo{M}^k$ that can be erroneous.
To understand the effect of \textit{temporal window length} $q$, on the prediction accuracy of \algo{lfpAlgo}, we again consider the linear ODE from \Cref{linODE}, and use different values of $q$, for $\kappa = 1$.
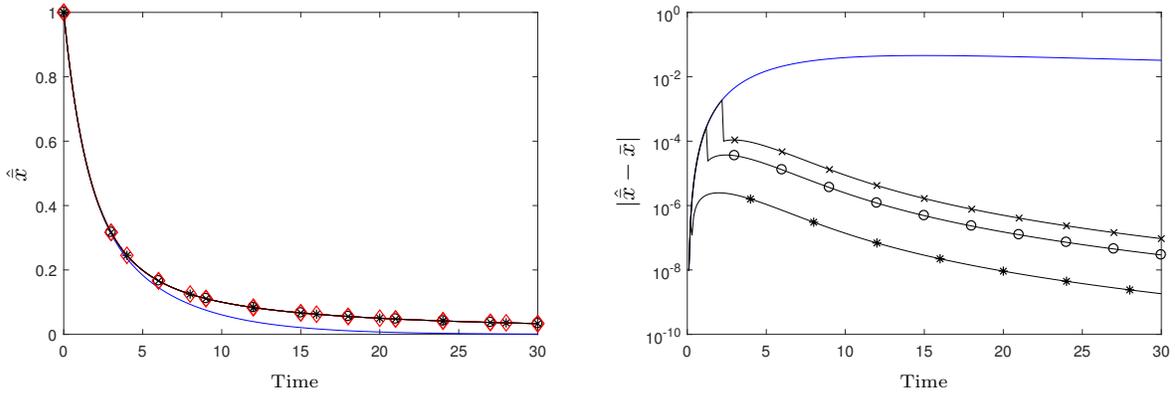
\begin{figure}[ht!]
\begin{center}
        \includegraphics[trim={2cm 0.1cm 2cm 0.2cm},clip,width=\textwidth]{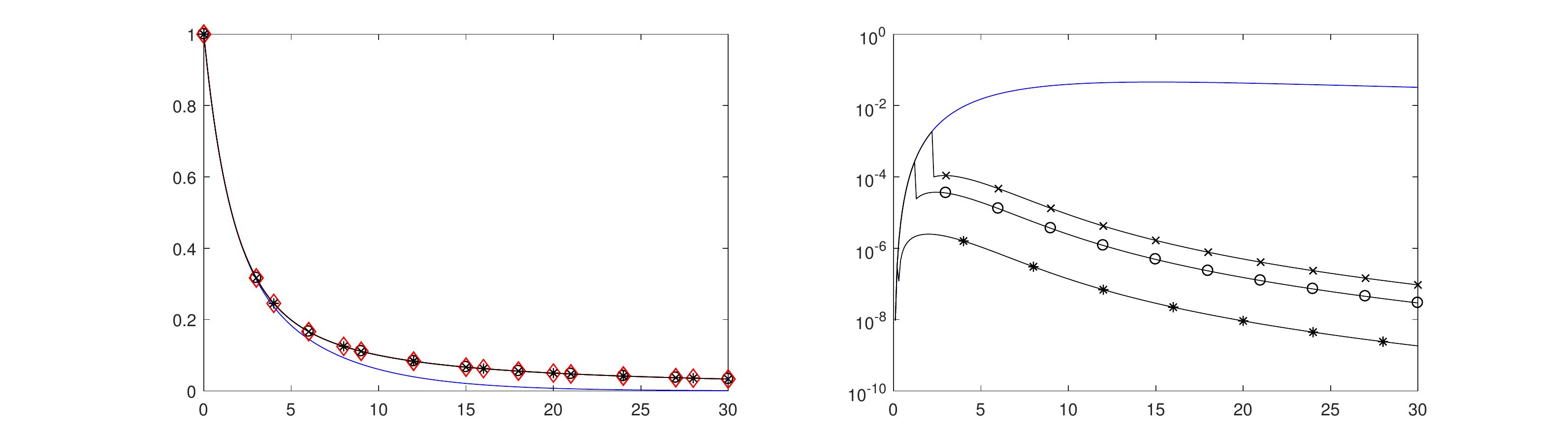}
\begin{picture}(0,0)
    \put(-220,85){\rotatebox{90} {{\small$\hat{\bar{x}}$}}}
    \put(10,75){\rotatebox{90} {{\small$|\hat{\bar{x}} - \bar{x}|$}}}
    \put(-120,7){{\scriptsize Time}}
    \put(118,7){{\scriptsize Time}}
\end{picture}
\caption{{\small Linear system - Effect of window length, $q$, on prediction accuracy of mean using \algo{lfpAlgo}. Analytical: red, GP formulation: blue, \algo{lfpAlgo}: black. $q= n(N+1)$(*), $5n(N+1)$(o), $10n(N+1)$(x), $\kappa = 1$.}}
\figlabel{LinODE_windowLength}
\end{center}
\end{figure}
As expected, from \fig{LinODE_windowLength} we observe that for $q=n(N+1)$ \algo{lfpAlgo} has the least error in predicted mean. As $q$ is increased, the prediction accuracy decreases.
\end{remark}

\section{Conclusions} \label{sec:conclusions}
Standard Galerkin projection (GP) and least squares (LS) approaches may demand higher order polynomial chaos (PC) approximations to achieve desired level of accuracy of moments estimated using a surrogate model.
In this paper, we presented modifications of GP and LS approaches with constraints, namely, constrained $\mathcal{L}_2$ and constrained $l_2$ formulations, which can match the first two estimated moments exactly with the true moments.
This enables us to reduce the approximation order of a PC expansion, hence, making the surrogate model simpler and computationally efficient without compromising the statistical accuracy.

We also presented a formulation that can be used to fit a piecewise linear surrogate model over finite intervals of time series data. It can also be used to propagate PC coefficients forward in time to predict the first two moments of a stochastic system with better accuracy. We expect these to be extremely useful in data-driven modeling, uncertainty quantification, and design of control algorithms for nonlinear stochastic systems. These applications will be explored in our future work.

\appendix
\section{Galerkin projection leads to unbiased minimum variance approximation} \label{sec:gpUnbiased}
Consider the stochastic ODE defined in \eqn{genODE}
\begin{align}
\xdot(t,\param) = \vo{f}(\x(t,\param)). \eqnlabel{genODE2}
\end{align}
Substituting the PC approximation of the solution from \eqn{XNODE} in \eqn{genODE2}, we get the residue
\begin{align}
\vo{e}(t,\param) := \vo{\dot{X}}\vo{\Phi}(\param) - \vo{f}\big(\X\vo{\Phi}(\param)\big). \eqnlabel{linres1}
\end{align}
Using the relation $\vec{\A\B\C} : = (\C^T\otimes\A)\vec{\B}$, \eqn{linres1} can be written as
\begin{align}
\notag \vo{e}(t,\param) &= \vec{\vo{e}(t,\param)} = \vec{\vo{\dot{X}}\vo{\Phi}(\param) - \vo{f}\big(\X\vo{\Phi}(\param)\big)} \\
& = (\phipt\otimes\I{n})\vec{\vo{\dot{X}}} - \vec{\vo{f}\big((\phipt\otimes\I{n})\vec{\vo{X}}\big)}. \eqnlabel{linres2}
\end{align}
In terms of $\xpc:=\vec{\X}$, \eqn{linres2} becomes
\begin{align*}
\vo{e}(t,\param) =  (\phipt\otimes\I{n})\xpcdot - \vo{f}\big((\phipt\otimes\I{n})\xpc\big).
\end{align*}
Setting the projection of $\vo{e}(t,\param)$ on $\phi_i(\param)$ to zero, i.e. $\E{\vo{e}(t,\param)\phi_i(\param)} = 0$ for $i=0,\cdots,N$. This results in the following deterministic ordinary differential equations,
\begin{align*}
\E{(\phip\otimes\I{n})(\phipt\otimes\I{n})}\xpcdot = \E{(\phip\otimes\I{n})\vo{f}\big((\phipt\otimes\I{n})\xpc\big)}.
\end{align*}
Using the relation $(\A\otimes\B)(\C\otimes\D) = (\A\C)\otimes(\B\D)$, we can simplify the above equation to
\begin{align}
\xpcdot = \left(\E{\phip\phipt}\otimes\I{n}\right)^{-1}\E{(\phip\otimes\I{n})\vo{f}\big((\phipt\otimes\I{n})\xpc\big)}.
\eqnlabel{gpcNonLinSurr}
\end{align}
Note that \eqn{genXpcODE} is a compact representation of \eqn{gpcNonLinSurr}, and the latter is a deterministic finite dimensional approximation of \eqn{genODE2}. If $\phi_i(\Delta)$ are orthogonal polynomials, then $\phi_0(\param) = 1$. Consequently, $\E{\vo{e}(t,\param)\phi_0(\param)} = 0 \implies \E{\vo{e}(t,\param)} = 0$,
i.e. the Galerkin projection (GP) leads to unbiased equation error. 

For given $\xpc$, we next derive $\xpcdot$ from an optimization perspective that minimizes $\E{\vo{e}^T(t,\param)\vo{e}(t,\param)}$, which can be written as
\begin{align*}
\E{\vo{e}^T(t,\param)\vo{e}(t,\param)} = \mathbb{E}\Bigg[ & \left[\left(\phipt\otimes\I{n}\right)\xpcdot - \vo{f}\big((\phipt\otimes\I{n})\xpc\big)\right]^T \\
& \left[\left(\phipt\otimes\I{n}\right)\xpcdot - \vo{f}\big((\phipt\otimes\I{n})\xpc\big)\right] \Bigg]
\end{align*}
\begin{align*}
= & \xpcdot^T\E{\left(\phip\phipt\right)\otimes\I{n}}\xpcdot - 2\xpcdot^T\E{(\phip\otimes\I{n})\vo{f}\big((\phipt\otimes\I{n})\xpc\big)}\\
+ & \E{\vo{f}\big((\phipt\otimes\I{n})\xpc\big)^T \vo{f}\big((\phipt\otimes\I{n})\xpc\big)}.
\end{align*}
Minimum is achieved when,
\begin{align*}
& \frac{\partial  \E{\vo{e}^T(t,\param)\vo{e}(t,\param)}}{\partial \xpcdot^T} = 0,
\end{align*}
or
\begin{align}
\xpcdot^\ast = \E{\left(\phip\phipt\right)\otimes\I{n}}^{-1}\E{(\phip\otimes\I{n})\vo{f}\big((\phipt\otimes\I{n})\xpc\big)},
\end{align}
which is identical to \eqn{gpcNonLinSurr}. Therefore, the GP formulation results in the $\mathcal{L}_2$ optimal approximation, and consequently the GP formulation is an unbiased minimum variance approximation.

Let us consider a special case in which the stochastic ODE considered in \eqn{genODE2} is linear in $\x$ with stochastic uncertainties in the system matrix as
\begin{equation}
\xdot(t,\param) = \Ap\x(t,\param),\eqnlabel{linDyn}
\end{equation}
where $\x:=\x(t,\param) \in\real^n$ and $\Ap \in \real^{n\times n}$.
In this case,
\begin{align*}
\vo{f}\big((\phipt\otimes\I{n})\xpc\big) = \Ap(\phipt\otimes\I{n})\xpc =  (\phipt\otimes\Ap)\xpc.
\end{align*}
Consequently, \eqn{gpcNonLinSurr} simplifies to
\begin{align}
\xpcdot = \left(\E{\phip\phipt}\otimes\I{n}\right)^{-1}\E{(\phip\phipt)\otimes\Ap)}\xpc.
\eqnlabel{gpcLinSurr}
\end{align}
Note that in \Cref{sec:lfp,sec:num_results}, the PC coefficients determined using this GP based approach are denoted by $\xgpc$.

\section*{Acknowledgments}
This work was supported by the National Science Foundation (NSF) under contract no.~1762825.

\bibliographystyle{unsrt}
\bibliography{accuratePC_arXive}

\end{document}